\documentclass[11pt]{amsart}
\usepackage{amssymb,mathrsfs,graphicx,extpfeil}
\usepackage{epsfig}
\usepackage{indentfirst, latexsym, amssymb, enumerate,amsmath,graphicx}
\usepackage{float}
\usepackage{colortbl}
\usepackage{epsfig,subcaption}
\usepackage{hyperref}
\usepackage{amsmath}
\title[Asymptotic behavior and control of a ``guidance by repulsion'' model]{Asymptotic behavior and control of a ``guidance by repulsion'' model}

\author{Dongnam Ko}  
\address{
\begin{flushleft}
Dongnam Ko
\end{flushleft}
\begin{flushleft}
DeustoTech, University of Deusto, 48007 Bilbao, Basque Country, Spain, \end{flushleft}
\begin{flushleft}
and
\end{flushleft}
\begin{flushleft}
Facultad de Ingenier\'ia, Universidad de Deusto, Avenida de las Universidades 24, 48007 Bilbao, Basque Country, Spain,
\end{flushleft}
\begin{flushleft}
dongnamko@deusto.es
\end{flushleft}
}

\author{Enrique Zuazua}
\address{
\begin{flushleft}
Enrique Zuazua
\end{flushleft}
\begin{flushleft}
Chair in Applied Analysis, Alexander von Humboldt-Professorship, 
\end{flushleft}
\begin{flushleft}
Department of Mathematics, Friedrich-Alexander-Universität Erlangen-Nürnberg,
91058 Erlangen, Germany,
\end{flushleft}
\begin{flushleft}
and
\end{flushleft}
\begin{flushleft}
DeustoTech, University of Deusto, 48007 Bilbao, Basque Country, Spain, \end{flushleft}
\begin{flushleft}
and
\end{flushleft}
\begin{flushleft}
Facultad de Ingenier\'ia, Universidad de Deusto, Avenida de las Universidades 24, 48007 Bilbao, Basque Country, Spain, 
\end{flushleft}
\begin{flushleft}
and
\end{flushleft}
\begin{flushleft}
Departamento de Matem\'aticas, Universidad Aut\'onoma de Madrid, 28049 Madrid, Spain. 
\end{flushleft}
\begin{flushleft}
enrique.zuazua@deusto.es
\end{flushleft}
}

\newtheorem{theorem}{Theorem}[section]
\newtheorem{lemma}{Lemma}[section]

\newtheorem{remark}{Remark}[section]

\newtheorem{definition}{Definition}[section]
\newcommand{\bbr}{\mathbb R}

\newcommand{\bu}{\mbox{\boldmath $u$}}
\newcommand{\bv}{\mbox{\boldmath $v$}}

\newcommand{\bx}{\mbox{\boldmath $x$}}

\begin{document}

\subjclass[2010]{} 

\thanks{\textbf{Acknowledgment.} The second author wishes to thank Ram\'on Escobedo and Aitziber Iba\~nez for fruitful discussions. This project has received funding from the European Research Council (ERC) under the European Union’s Horizon 2020 research and innovation programme (grant agreement No. 694126-DyCon).
The work of the first author has been supported by CNCS-UEFISCDI Grant No. PN-III-P4-ID-PCE-2016-0035.
The work of the second author has been funded by the Alexander von Humboldt-Professorship program, the European Union’s Horizon 2020 research andinnovation programme under the Marie Sklodowska-Curie grant agreement No.765579-ConFlex, grant MTM2017-92996-C2-1-R COSNET of MINECO (Spain), ELKARTEK project KK-2018/00083 ROAD2DC of the Basque Government, ICON of the French ANR and Nonlocal PDEs: Analysis, Control and Beyond, AFOSR Grant FA9550-18-1-0242.}

\begin{abstract}
We model and analyze a herding problem, where the drivers try to steer the evaders' trajectories while the evaders always move away from the drivers. This problem is motivated by the \emph{guidance-by-repulsion} model \cite{escobedo2016optimal}, where the authors answer how to control the evaders' positions and what is the optimal maneuver of the drivers.
First, we obtain the well-posedness and the long-time behavior of the one-driver and one-evader model, assuming of the same friction coefficients. In this case, the exact controllability is proved in a long enough time horizon.
We extend the model to the multi-driver and multi-evader case, and develop numerical simulations to systematically explore the nature of controlled dynamics in various scenarios. 
The optimal strategies turn out to share a common pattern to the one-driver and one-evader case: the drivers rapidly occupy the position behind the target, and want to pursuit evaders in a straight line for most of the time. Inspired by this, we build a feedback strategy which stabilizes the direction of evaders.

\end{abstract}
\date{\today}
\maketitle 

\section{Introduction}
\setcounter{equation}{0}
Interactions between different groups of individuals are often observed in collective behavior models, such as the schooling of fish and whales or the herding of sheep and dogs. Concerning the herding (or hunting) problem \cite{coppinger2015dogs,reynolds1987flocks}, this kind of systems have been studied, where there are herding \emph{evaders} (sheep) interacting with their \emph{drivers} (shepherd dogs). There is a vast literature focusing on various related topics, 
for example, understanding and simulating real data of shepherd dogs \cite{strombom2014solving}, 
analyzing the optimal number of predators in the chase-and-escape model \cite{vicsek2010statistical} 
and studying how cooperation arises and enhances efficiency in hunting problems \cite{escobedo2014group}.

In the study of these phenomena, control theory can be a natural strategy to tackle the herding problem \cite{burger2014partial}.
In the recent works \cite{muro2011wolf,wang2015motion}, effective chase strategies are designed for hunting problems. In \cite{pierson2017controlling}, feedback formation is suggested to collect sheep in a small area. In \cite{bongini2017mean}, the well-posedness for optimal control problems is established on transport equations of the herd, coupled with ordinary differential equations (ODEs) of the dogs.

In order to herd the evaders, the control system relies on a secondhand control through the drivers. Since we cannot directly control the evaders, the first-order linearized model does not give us useful information. It is difficult, or even impossible, to control the evaders' positions in a linearized model, resulting only on partial controllability results that are then hard to extrapolate to the nonlinear model. For this reason, a genuinely nonlinear analysis of the models is needed. 

On the other hand, the \emph{guidance-by-repulsion} model \cite{escobedo2016optimal} suggests an analytic approach to model the interactions between one driver and one evader, where the driver steers the evader's final position by combining two strategies; the pursuit and circumvention maneuver. More precisely, the driver naturally follows the evader at a safe distance along linear trajectories (the pursuit), but it can perform rotational motion to change the escaping direction of the evader (circumvention maneuver). 

Our interest lies in modeling and analyzing these dynamics with many drivers and many evaders. A systematic computational analysis of this issue is also developed. In order to construct models mimicking the interaction between shepherd dogs and sheep, we assume that the interactions should satisfy the following regulations:
\begin{itemize}
\item The drivers follow the evaders but cannot be arbitrarily close to them because of animal conflict, chemical repulsions, etc.
\item The drivers also interact between each other in order to avoid collisions.
\item The evaders escape from the drivers but also seek to flock together.
\item When a driver is close to the evaders, it can display a circumvention maneuver around the evaders that forces
them to change their direction.
\item Thus, by adjusting the onset and
offset of the circumvention maneuver, the evaders can be driven toward a desired target position.
\item The drivers can stop the pursuit motion and wait until the evaders escape and flock together again.
\item For simplicity, we assume that each driver knows and follows the barycenter of the evaders. It is a non-trivial problem to determine what the drivers see and what motivates them to move, act and interact.
\end{itemize}

With these principles as basis, and inspired by the existing abundant literature, we formulate a new model as a control system. After discussions on the analytical properties of this model, we perform a number of computational experiments in order to explore the efficiency of the control strategies developed, and inspired by these results, we build a feedback control to herd the evaders.

A priori, we do not set up any collaborative strategy between drivers, and each of them establishes a relation with respect to the crowd of evaders as if it were the sole driver.
However, the optimal control strategy turns out to show an emerging complex control dynamics, assigning to each driver a specific role that is not easy to anticipate from the formulation of the control problem. 

The rest of this paper is organized as follows. In the consecutive two subsections, we introduce the formulation and main results for the controlled dynamics (Section \ref{sec:1.1}) and discuss related works on the herding problem (Section \ref{sec:1.2}). 
In Section \ref{sec:simple}, we state the well-posedness and controllability for the simplified model, which consists of one driver and one evader.
The proofs are presented in Section \ref{sec:control} and they make use of the Lyapunov function method. Section \ref{sec:multi} is devoted to simulations for the optimal control strategies in the multi-driver and multi-evader model with a comparison to the simplified one. In Section \ref{sec:feedback}, the feedback control is built and explained with its simulations. Finally, in Section \ref{sec:con}, some final remarks and open problems are presented.

\subsection{Problem formulation and control strategies}\label{sec:1.1}

In order to formulate the herding problem, we consider an interacting particle system with $M$ drivers and $N$ evaders.
Let $\bu_{ei}$ be the position of the $i$-th evader in $\mathbb R^2$ for $i = 1,\ldots,N$, and $\bu_{dj}$ be that of the $j$-th driver for $j = 1,\ldots,M$.  
Then, the guidance-by-repulsion model can be described by the interactions among those agents as follows.
\begin{itemize}

\item Each agent follows Newtonian dynamics with friction, and the interactions are based on their relative positions. 

\item The $i$-th evader gets a repulsive force from the $j$-th driver,
\[ -\frac{1}{M}\sum_{j=1}^Mf_e(|\bu_{dj} - \bu_{ei}|)(\bu_{dj} - \bu_{ei}), \]
where $f_e$ is a nonnegative function, $f_e(r) \to \infty$ when $r \to 0$ and $f_e(r) \to 0$ as $r \to \infty$. (for example, see \eqref{fd_fe} or Figure \ref{fig:fdfefr}).

\item The interaction among evaders enhances the flocking of them. The total interaction force on the $i$-th evader is given by
\[ \frac{1}{N}\sum_{k \neq i} \psi_e(|\bu_{ek} - \bu_{ei}|)(\bu_{ek} - \bu_{ei}), \]
where $\psi_e(r)$ is an attractive-repulsive function, which is positive if $r$ is large and negative if $r$ is small (see also \eqref{fd_fe}). Hence, the evaders flock but they keep a positive distance between them.

\item The $j$-th driver $\bu_{dj}$ pursuit the barycenter of evaders, $\bu_{ec}$, by the force
\[ -\kappa^p_j(t)f_d(|\bu_{dj} - \bu_{ec}|)(\bu_{dj} - \bu_{ec}), \quad \bu_{ec} := \frac{1}{N}\sum_{k=1}^N \bu_{ek}, \]
where $f_d : \mathbb R_+ \to \mathbb R$ is a nonnegative bounded function and $\kappa^p_j$ is a bounded control function. The pursuit occurs when the values of $f_d$ and $\kappa^p_j$ are positive.

\item The drivers may perform the circumvention maneuver in the perpendicular direction of $(\bu_{dj} - \bu_{ec})$: for the $j$-th driver, this can be represented as
\[ \kappa^c_j(t)(\bu_{dj} - \bu_{ec})^\perp, \quad (u_1,u_2)^\perp := (-u_2,u_1) \text{ in } \mathbb R^2, \]
where $\kappa^c_j$ is a bounded control function.

\item In order to avoid collisions among drivers, the $j$-th driver get forces from other drivers,
\[ \frac{1}{M}\sum_{l \neq j}\psi_d(|\bu_{dl} - \bu_{dj}|)(\bu_{dl} - \bu_{dj}), \]
where $\psi_d(r)$ takes positive values if $r$ is small in order to repel other drivers (see also \eqref{fd_fe}).
\end{itemize}

Putting all these together, the whole driver-evader interactions can be represented as a nonlinear system of ODEs as follows:
\begin{equation}\label{GBR_general}
\begin{aligned}
\begin{cases}
\displaystyle \ddot{\bu}_{dj} = -\kappa_j^p(t)f_d(|\bu_{dj} - \bu_{ec}|)(\bu_{dj} - \bu_{ec}) -\frac{1}{M}\sum_{l \neq j} \psi_d(|\bu_{dl} - \bu_{dj}|)(\bu_{dl} - \bu_{dj}) - \nu_{dj} \dot{\bu}_{dj}\\
\hspace{1cm}+ \kappa_j^c(t) (\bu_{dj} - \bu_{ec})^\perp, \\
\displaystyle \ddot{\bu}_{ei} =  - \frac{1}{M} \sum_{j=1}^M f_e(|\bu_{dj} - \bu_{ei}|)(\bu_{dj} - \bu_{ei}) -\frac{1}{N}\sum_{k\neq i} \psi_e(|\bu_{ek} - \bu_{ei}|)(\bu_{ek} - \bu_{ei}) - \nu_{ei} \dot{\bu}_{ei}, \\
\displaystyle \bu_{dj}(0) = \bu_{dj}^0,~ \bu_{ei}(0) = \bu_{ei}^0,~ \dot{\bu}_{dj}(0) = \bv_{dj}^0, ~ \dot{\bu}_{ei}(0) = \bv_{ei}^0,\quad i=1,\ldots,N,\quad j=1,\ldots,M,
\end{cases}
\end{aligned}
\end{equation}
where $\nu_{dj}$ and $\nu_{ei}$ are constant friction coefficients and the functions $f_d$, $f_e$, $\psi_d$, and $\psi_e$ are locally Lipschitz functions.

Our objective of the control in \eqref{GBR_general} would be to herd the evaders $\bu_{ei}$ using the drivers $\bu_{dj}$; we want to gather, drive and trap the evaders into a specific area.
\begin{definition}\label{D1.1}
For a given spatial set $D$, the \emph{guidance-by-repulsion} problem \eqref{GBR_general} is controllable to $D$ if there exist
\[ t_f>0,~ \kappa^p_j,~ \kappa^c_j \in L^\infty([0,t_f],\mathbb R),~j=1,\ldots,M,\]
such that the solution of \eqref{GBR_general} satisfies 
 \[\bu_{ei}(t_f) \in D,~ \forall i = 1,\ldots,N. \]
\end{definition}

Bilinear controls $\kappa_j^p(t)$ and $\kappa_j^c(t)$ enter in the system in an open-loop manner, where they represent the strength of the pursuit motion and the rotational circumvention motion, respectively. 
For example, when $\kappa_j^p(t) = 1$ and $\kappa_j^c(t) = 0$, the drivers get forced by $f_d$, so that they track the evaders in the direction of $(\bu_{dj}-\bu_{ec})$, referred to as the \emph{pursuit mode}. On the other hand, if $\kappa_j^p(t) = 1$ and $\kappa_j^c(t) = \kappa_j^c \neq 0$, we refer to it as the \emph{circumvention mode}, since the driver moves into a perpendicular direction to steer the evader's direction. Furthermore, the drivers stop tracking the evaders when $\kappa_j^p(t) = 0$ and $\kappa_j^c(t) = 0$, which is the \emph{release mode}. Then, the evaders will escape from the drivers and flock together again.

To gain some understanding of the controlled dynamics above, here we present simple simulations regarding the pursuit and circumvention modes.
Figure \ref{fig:multi} represents simulations of 1 driver and 5 evaders with the parameters $\nu = \nu_{d1} = \nu_{ei} = 2$ for $i=1,\ldots,5$, and the following nonlinearities $f_d$, $f_e$, $\psi_d$ and $\psi_e$:
\begin{equation}\label{fd_fe}
\begin{aligned}
f_d(r) = 1,\quad f_e(r) = \frac{1}{r^2},\quad \psi_d(r) = \frac{1}{2r^4}
\quad\text{and}\quad \psi_e(r) = 10\left(\frac{(0.1)^2}{r^2}-\frac{(0.1)^4}{r^4}\right).
\end{aligned}
\end{equation}

The first simulation (left) shows the trajectories in the circumvention mode ($\kappa^p_1(t) = 1$, $\kappa^c_1(t) = 0$) for $t \in [0,15]$, where we can see rotational motion of the evaders after the driver and evaders are close enough.
The second one (right) follows the circumvention mode for $t \in [0,10]$, however, changes to the pursuit mode ($\kappa^p_1(t) = 1$, $\kappa^c_1(t) = 0$) after $t=10$. By adjusting to the pursuit mode after some time, we may drive the evaders to an appropriate direction we want. 

\begin{figure}[ht]
  \centering
  {
    \includegraphics[width=0.45\textwidth]{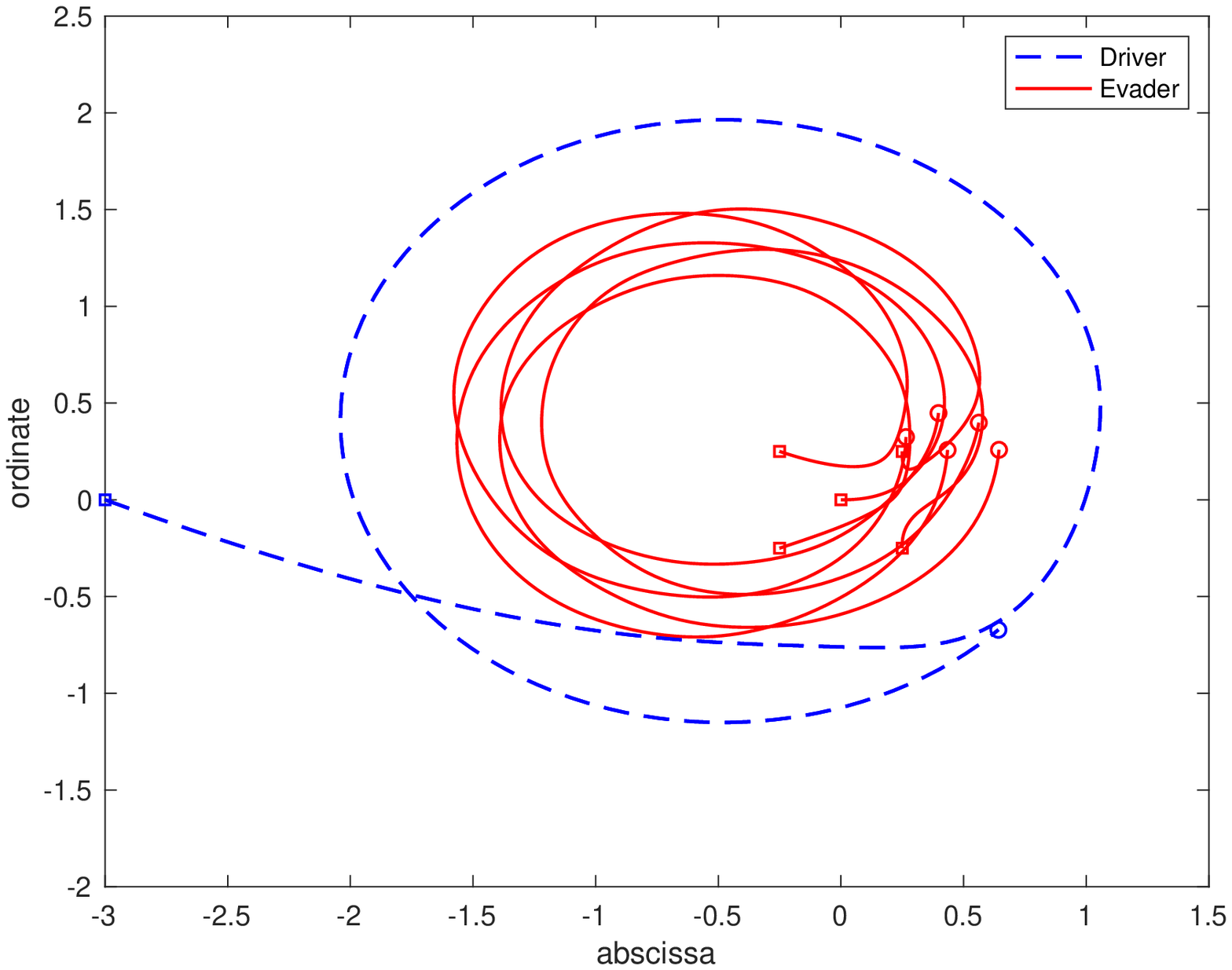}
    \includegraphics[width=0.45\textwidth]{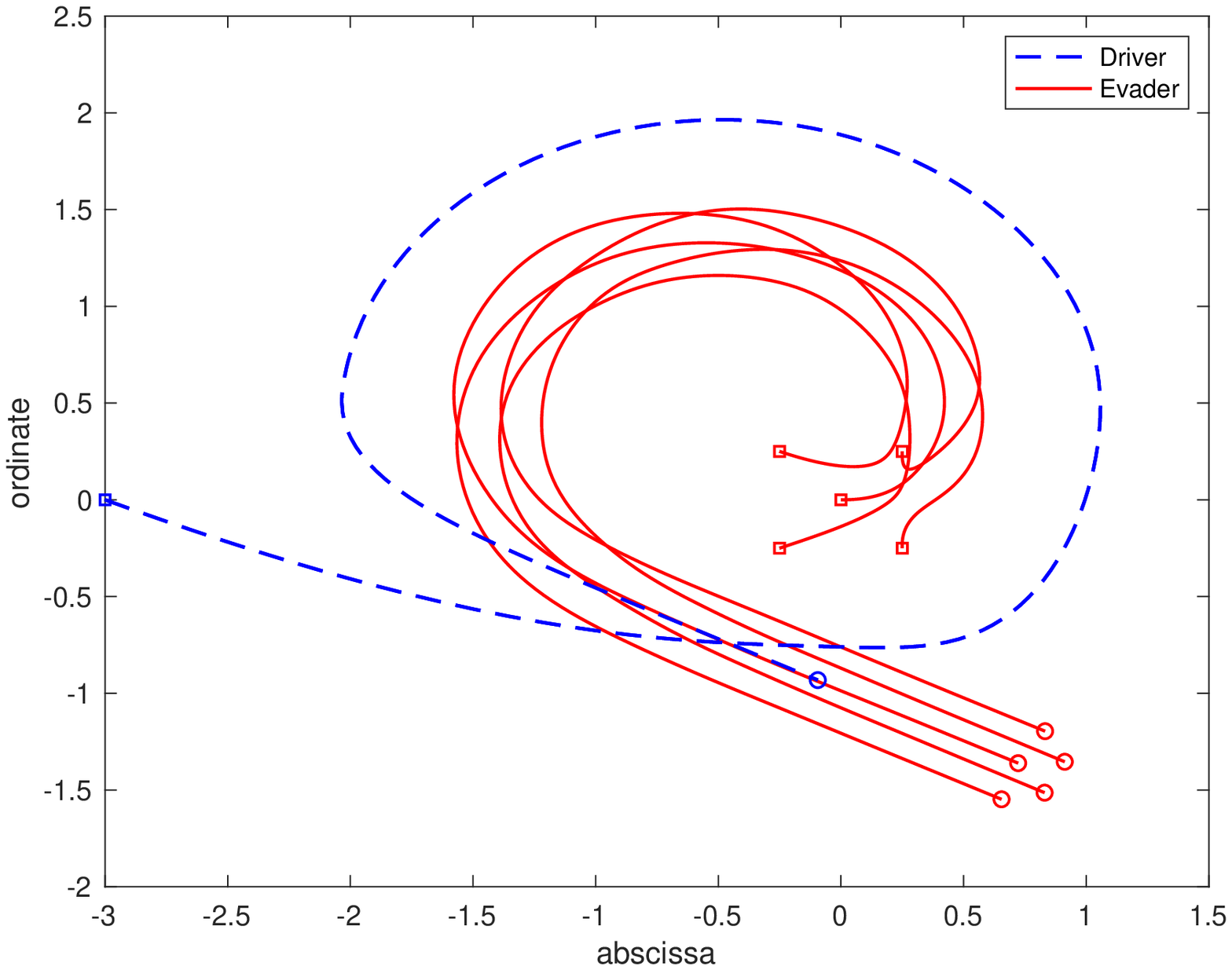}
  }
  \caption{The trajectories of \eqref{GBR_general} with 1 driver and 5 evaders. The square and circle marks represent the initial and final positions, respectively. A nonzero circumvention control $\kappa^c_1(t) = 1$ for $t \in [0, 15]$ leads to the rotational motion eventually (left). By turning off the circumvention control at $t=10$, linear pursuit motion arises (right). In this way, we may steer the evaders to a desired area.}
  \label{fig:multi}
\end{figure}

In this paper, we focus on the dynamical features and control strategies of the guidance-by-repulsion model \eqref{GBR_general}. On the one hand, in order to analyze the interaction between the drivers and evaders, we study a simplified model with one driver and one evader. 
For this model, we discuss its well-posedness and the asymptotic motion of the driver and evader in each mode. As in Figure \ref{fig:multi}, linear and rotational trajectories arise respectively in the pursuit and the circumvention mode. These time-asymptotic convergences lead us to prove the exact controllability of the evader's position with a long enough time horizon.

On the other hand, from a control perspective, we simulate optimal open-loop control strategies on the multi-driver and multi-evader model. We use the gradient method to find optimal controls and controlled trajectories. 
Among the simulations, a simple pattern appears: the drivers start with the circumvention motion to steer the directions of the evaders, and then switch smoothly to the pursuit motion toward the target point.

This pattern motivated us to build a feedback strategy, where each driver can decide proper controls $\kappa^p_j(t)$ and $\kappa^c_j(t)$ from the information of the target point, the barycenter of evaders and the diameter of evaders. When there are more than one driver, this feedback strategy naturally make the drivers to build a formation behind the herd of evaders and push the evaders as a team.

\subsection{Related works on the herding problem}\label{sec:1.2}
As we introduced in the beginning, there have been a lot of researches to understand the herding problem and simulate its dynamics. In the viewpoint of control theory, here we list several related works:

\begin{itemize}
\item In \cite{king2012selfish,strombom2014solving}, a time-discrete model has been studied with one driver and many evaders in order to explain the real-world data of shepherd dogs. The dog is designed to track an abnormal sheep escaping from the herd, and force it to move toward the center of sheep. This idea was extended to a multi-driver case in \cite{lee2017autonomous}, where the drivers try to control the nearest evader if the evaders keep close each other.

\item In the context of automation design, in \cite{lien2004shepherding}, they classify the configuration of the evaders' positions, and provide a specific strategy of one driver for each situation. They suggest that the driver need to know the sight of evaders and should avoid their personal space for herding.

\item From the viewpoint of the large population limit of evaders, the optimal control of the herding problem is formulated in \cite{burger2016controlling,pinnau2018interacting}. In these papers, the density function of the evaders is described in the transport equation when the interactions are bounded and smooth. Under this assumption, the optimal control problem is well-posed, and the evaders eventually accumulates to one point.

\item Several feedback strategies are suggested in \cite{pierson2017controlling}, where the drivers first surround the evaders in a small area and escort them to the target point. This extends the results of \cite{muro2011wolf,wang2015motion}, where they control the drivers to make a polygon formation and surround the evaders.

\item Feedback controls for collecting separated evaders have been studied in \cite{licitra2017single}. Their strategy is to drive evaders one by one to a given point, and they proved that the convergence of evaders' positions is exponential.
\end{itemize}

Compared to these works, our novelty comes from the equations of motion for the drivers. In \eqref{GBR_general}, the dynamics of drivers are governed by simple interaction rules and restricted so that we can not control them freely. Instead of manipulating each driver by hand, we determine the pursuit and circumvention modes of the drivers by two control functions $\kappa^p_j(t)$ and $\kappa^c_j(t)$. This also enables us to easily understand simulations of the optimal control, and we can build feedback controls from the idea of optimal control strategies. 

\vspace{1em}
\section{A simplified Model: One driver and one evader}\label{sec:simple}

One of the standard methods understanding a collective model is to analyze the behavior of a small number of particles \cite{bellomo2014sociology,bellomo2013complexity,vicsek2010statistical}, where the emergent dynamics often arise with simple conditions related to the micro-scale interactions. To enhance the intuition on the herding strategies, we first focus on the simplified model with one driver and one evader.

Let $\bu_d$ and $\bu_e \in \bbr^2$ be the positions of the driver and evader, respectively, and $\kappa^p(t) := \kappa_1^p(t)$ and $\kappa^c(t) := \kappa_1^c(t)$ be the pursuit and circumvention control functions of the driver. From \eqref{GBR_general}, the dynamics can be rewritten as
\begin{equation}\label{GBR_simple}
\begin{aligned}
\begin{cases}
\displaystyle\dot {\bu}_d = \bv_d,\quad \dot {\bu}_e = \bv_e,\quad \bu := \bu_d-\bu_e,\quad \dot{\bu}=\bv,\\
\displaystyle \dot {\bv}_d = -\kappa^p(t)f_d(|\bu|)\bu + \kappa^c(t) \bu^\perp - \nu \bv_d,\\
\displaystyle \dot {\bv}_e = -f_e(|\bu|)\bu - \nu \bv_e,\\
\displaystyle \bu_d(0) = \bu_d^0,~ \bu_e(0) = \bu_e^0,~ \bv_d(0) = \bv_d^0,~ \bv_e(0) = \bv_e^0,
\end{cases}
\end{aligned}
\end{equation}
where we assumed the same dissipation coefficients $\nu = \nu_d = \nu_e$ for the driver and evader.

The assumption of $\nu_d = \nu_e$ will play a critical role though it is a technical requirement not desirable from practical viewpoints. The difference of the friction generate oscillatory dynamics on the relative distance between the driver and evader. In contrast, under the same dissipation assumption, note that the equation of the relative position $\bu$ in \eqref{GBR_simple} can be described in a separated and closed equation:
\begin{equation}\label{GBR_rel}
\begin{aligned}
\ddot{\bu} + (\kappa^p(t)f_d(|\bu|) - f_e(|\bu|))\bu + \nu \dot{\bu} = \kappa^c(t) \bu^\perp,
\end{aligned}
\end{equation}
which resembles the equation of the harmonic oscillators.

In this section, we discuss well-posedness and asymptotic behavior of the model \eqref{GBR_simple} and \eqref{GBR_rel}. By combining asymptotic solutions, we can build an off-bang-off control to show that there exists a controlled trajectory which leads $\bu_e(t_f) = \bu_f$.

\subsection{Global well-posedness}\label{sec:simple.global}

 Since the interaction between the driver and the evader is attractive and repulsive depending on the distance (for example, $f_d$ and $f_e$ in \eqref{fd_fe}), we assume that the relative force $f(r) := f_d(r) - f_e(r)$ satisfies the following condition with a constant $r_p > 0$:
\begin{equation}\label{eq_f}
f(r) = \begin{cases} < 0 \quad\text{for }~ 0 < r < r_p, \\
\geq 0 \quad \text{for }~ r \geq r_p \end{cases}~ \text{with}\quad f'(r_p) > 0.
\end{equation}
This assumption on $f$ correspond to the Van der Waals type forces which exclude collisions between individuals in \cite{vicsek2010statistical}. Then, the equation \eqref{GBR_rel} with $\kappa^p(t)=1$ is a damped oscillator model with respect to the potential energy
\[ P(\bu) := \int_{r_p}^{|\bu|} rf(r)dr,\]
affected by an additional perpendicular force $\kappa^c(t)\bu^\perp$.

Moreover, for the well-posedness of the system, we suggest the following conditions of singular repulsive potential with a constant $\gamma_m >0$:
\begin{equation}\label{f_g}
\begin{aligned}
&f_d \in L^\infty(\mathbb R_+,\mathbb R),\quad \lim_{r \to \infty} f_d(r) = \gamma_m >0,\\
&0\leq f_e(r) <\infty ~\text{for}~ 0 < r < \infty,\quad \int_0^{r_p} rf_e(r)dr = \infty,\quad \text{and}\quad \lim_{r \to \infty} f_e(r) = 0.
\end{aligned}
\end{equation}
Then, $P$ becomes an unbounded potential, $P((0,0))=\infty$, and it grows quadratically as $\bu$ increases. 
For example, the interaction forces $f_d(r)$ and $f_e(r)$ suggested in \eqref{fd_fe} satisfy \eqref{eq_f} and \eqref{f_g}.

\begin{figure}[ht]
  \centering
    \includegraphics[width=0.9\textwidth]{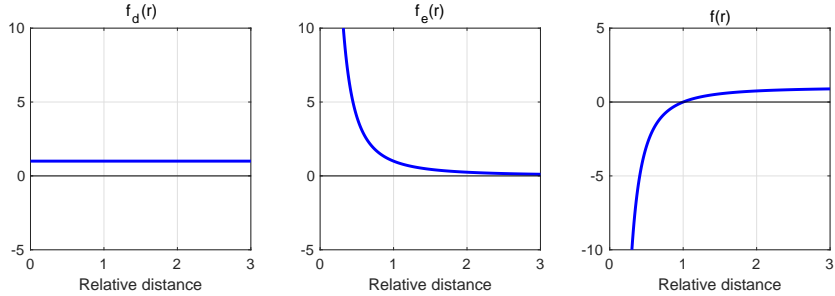}
  \caption{The graphs of interaction forces, $f_d(r)$, $f_e(r)$ and $f(r)=f_d(r)-f_e(r)$ in \eqref{fd_fe}, drawn only for $0 \leq r \leq 3$ and $-10\leq f(r) \leq 10$.}
  \label{fig:fdfefr}
\end{figure}

\begin{theorem}\label{T2.1}
Suppose that $f_d$ and $f_e$ satisfy the conditions \eqref{eq_f} and \eqref{f_g}.
Then, if the control functions $\kappa^p(t)$ and $\kappa^c(t)$ are uniformly bounded, then the global existence of the solution to \eqref{GBR_simple} is guaranteed. In addition, the trajectory of $\bu(t)$ remains away from $(0,0)$ in a finite time.

Moreover, if we additionally assume that the controls are constant and satisfy
\[ \kappa^p(t) \equiv \kappa^p >0,\quad \kappa^c(t) \equiv \kappa^c \quad\text{and}\quad |\kappa^c| < \nu\sqrt{\kappa^p\gamma_m}, \]
for $\gamma_m$ in \eqref{f_g}, then $|\bu(t)|$ is uniformly bounded from above and below along the whole time.
\end{theorem}

\begin{remark}\label{R2.1}
For a nonzero constant pursuit control $\kappa^p(t) \equiv \kappa^p$, we may assume $\kappa^p = 1$ without loss of generality. We may use $\kappa^pf_d(r)$ instead of $f_d(r)$ since this is still a bounded function depending only on $r$.
\end{remark}

The proof of the well-definedness is in Section \ref{sec:control} and uses energy methods as in other collective dynamics model \cite{carrillo2017sharp,cucker2010avoiding}. In particular, if $P$ does not blow-up in a finite time, then $\bu(t)$ cannot hit $(0,0)$ and the solution $(\bu_d,\bu_e)$ is well-defined.

\subsection{Asymptotic motion under constant controls}\label{sec:steady}

 After the well-posedness in Theorem \ref{T2.1}, the next question would be the asymptotic motion of $\bu_d(t)$ and $\bu_e(t)$. From the simulation in Figure \ref{fig:multi}, we expect that the dynamics of \eqref{GBR_simple} eventually tends to linear motion in the pursuit and release mode and rotational motion in the circumvention mode. In the following subsections, we suggest a linear or rotational solution on each mode, which represents possible asymptotic behavior. This analysis is based on the relative position $\bu(t)$ and its equation of motion \eqref{GBR_rel}.

\subsubsection{Case 1 : The pursuit mode}\label{sec:steady_0}

Suppose that the system evolves in the pursuit mode, i.e., $\kappa^p(t) \equiv 1$ and $\kappa^c(t) \equiv 0$. We want to find a solution representing linear motion, which may be a possible asymptotic solution in the pursuit mode.

 From the dissipative property of \eqref{GBR_rel}, we may expect the velocity $\dot{\bu}$ and acceleration $\ddot{\bu}$ vanish eventually. This equilibrium solution $\bar\bu(t)$ of the relative position follows
\begin{equation*}
\begin{aligned}
f(|\bar\bu(t)|) = 0,
\end{aligned}
\end{equation*}
then we have
\[ \bar \bu (t) \equiv \bu^* \in \mathbb R^2\quad\text{where}\quad \bu^* = r_p(\cos\phi_0, \sin\phi_0),  \]
where $r_p$ is defined in \eqref{eq_f} and $\phi_0$ is a constant in $[0,2\pi)$.

Next, we would like to determine asymptotic motion of the driver and evader when the relative position is given by $\bar\bu(t) \equiv \bu^*$. From the equation of motion \eqref{GBR_simple}, the corresponding positions $\bu_d$ and $\bu_e$ should satisfy
\begin{equation*}
\begin{aligned}
\ddot {{\bu}}_d + \nu \dot{\bu}_d = -f_d(|\bu^*|)\bu^*
\quad\text{and}\quad
\ddot {{\bu}}_e + \nu \dot{\bu}_e = -f_e(|\bu^*|)\bu^*,
\end{aligned}
\end{equation*}
where the right-hand sides are the same constant vector, $-f_d(|\bu^*|)\bu^* = -f_e(|\bu^*|)\bu^*$. Note that these are second order damping motions with constant external forces.
Then, we may conclude that $\bv_d(t)$ and $\bv_e(t)$ converge  exponentially to the constants $-f_d(|\bu^*|)\bu^*/\nu$, and then the positions tend to linear motion.

Therefore, $\bu_d(t)$ and $\bu_e(t)$ would converge to the following solutions,
\begin{equation}\label{steady_0}
\begin{aligned}
&\bar {\bu}_d(t) = -\frac{f_d(\bu^*)\bu^*}{\nu}t + \bu_d^* \quad\text{and}\quad
\bar {\bu}_e(t) = -\frac{f_d(\bu^*)\bu^*}{\nu}t + \bu_e^*,
\end{aligned}
\end{equation}
with constants $\bu_e^* \in \mathbb R^2$ and $\bu_d^* = \bu_e^* + \bu^*$. Hence, there is a family of linear motion \eqref{steady_0} with respect to the undetermined parameters $\phi_0 \in [0,2\pi)$ and $\bu_e^* \in \mathbb R^2$.

Note again that the trajectories \eqref{steady_0} are the asymptotic solutions when the relative position $\bu(t)$ is stationary. 
We will call it as the \emph{pursuit dynamics}, where the driver follows the evader with the same constant velocity in the direction of $\bu^*$. Hence, they will move in a line and diverge to the infinity point.

\subsubsection{Case 2 : The circumvention mode}\label{sec:steady_1}

While we found a linear motion in the pursuit mode, we consider a solution $\bar \bu(t)$ representing rotational motion in the circumvention mode.

 Let $\kappa^p(t)\equiv 1$ and $\kappa^c(t)\equiv \kappa^c$ with nonzero constant $\kappa^c$. The presence of a perpendicular force $\kappa^c\bu^\perp$ will rotate the relative position $\bu(t)$ with the help of the friction $\nu \bv$, where $f(|\bu|)\bu$ acts as a centripetal force.

Instead of an equilibrium of $\bu(t)$, here we start with an ansatz of rotational motion:
\[\bar \bu(t) = r_c(\cos (w_c t+\phi_1),\sin (w_c t+\phi_1), \]
for some constants $\phi_1$, $w_c$ and $r_c >0$.
By putting this equation into \eqref{GBR_rel}, we have
\begin{equation*}
\begin{aligned}
-(w_c)^2 \bu(t) + f(r_c)\bu(t) + \nu w_c \bu(t)^\perp = \kappa^c \bu(t)^\perp.
\end{aligned}
\end{equation*}
Since $\bu(t)$ and $\bu(t)^\perp$ are perpendicular, we classify the above terms with their directions, $\bu$ or $\bu^\perp$, which leads to the following compatibility conditions:
\[ f(r_c) = (w_c)^2 \quad\text{and}\quad \nu w_c = \kappa^c. \]
From the assumptions \eqref{eq_f} and \eqref{f_g} on $f$, there exists a solution $r_c>0$ to the equation
\begin{equation}\label{r_c}
f(r_c) = \Big(\frac{\kappa^c}{\nu}\Big)^2,
\end{equation}
if $\kappa^c$ is small (in particular, $|\kappa^c| < \nu\sqrt{\gamma_m}$). Hence, there is a rotational solution $\bar\bu(t)$,
\begin{equation}\label{steady_1u}
\bar \bu(t) = r_c \left( \cos \Big(\frac{\kappa^c}{\nu}t+\phi_1\Big),\sin \Big(\frac{\kappa^c}{\nu}t+\phi_1\Big) \right).
\end{equation}

Now we want to determine asymptotic solutions $\bu_d(t)$ and $\bu_e(t)$ with this $\bar\bu(t)$. From \eqref{GBR_simple}, we have
\begin{equation*}
\begin{aligned}
\ddot {{\bu}}_d = -f_d(r_c)\bar \bu + \kappa^c \bar \bu^\perp - \nu \dot{\bu}_d\quad\text{and}\quad
\ddot {{\bu}}_e = -f_e(r_c)\bar \bu - \nu \dot{\bu}_e.
\end{aligned}
\end{equation*}
From the evader's equation, $\bu_e$ experiences a periodic external force from $\bar \bu$ with a dissipation $-\nu\dot{\bu}_e$. Therefore, $\bu_e$ will tend to draw a circular trajectory with the same frequency as $\bar\bu$. In the same way, $\bu_d$ also tends to a rotational motion. From \eqref{GBR_simple} and $\bu_d(t) - \bu_e(t) = \bar\bu(t)$, the asymptotic solutions $\bar\bu_d(t)$ and $\bar\bu_e(t)$ are determined as
\begin{equation}\label{steady_1}
\begin{aligned}
\bar \bu_d(t) &= r_d \left( \cos \Big(\frac{\kappa^c}{\nu}t + \phi_d\Big),\sin \Big(\frac{\kappa^c}{\nu}t + \phi_d\Big) \right) + \bu^*_c \quad\text{and}\quad \\
\bar \bu_e(t) &= r_e \left( \cos \Big(\frac{\kappa^c}{\nu}t + \phi_e\Big),\sin \Big(\frac{\kappa^c}{\nu}t + \phi_e\Big) \right) + \bu^*_c,
\end{aligned}
\end{equation}
for some point $\bu^*_c \in \mathbb R^2$, and the constants $r_d$, $r_e$, $\phi_d$, $\phi_e$ satisfy
\begin{equation*}
\begin{aligned}
&r_d = \frac{\sqrt{f_d(r_c)^2+(\kappa^c)^2 }}{\sqrt{(\kappa^c/\nu)^4+(\kappa^c)^2}} r_c,\quad
r_e = \frac{f_e(r_c)}{\sqrt{(\kappa^c/\nu)^4+(\kappa^c)^2}} r_c,\\
&\phi_d = \phi_1 - \arctan\frac{\nu^2}{\kappa^c} - \arctan \frac{\kappa^c}{f_d(r_c)} \quad\text{and}\quad
\phi_e = \phi_1 - \arctan\frac{\nu^2}{\kappa^c},
\end{aligned}
\end{equation*}
with $\phi_1 \in [0,2\pi)$ from $\bar \bu(t)$. As in Case 1, the solutions \eqref{steady_1} of \eqref{GBR_simple} form a family of rotational motion where $\phi_1 \in [0,2\pi)$ and $\bu_c^* \in \mathbb R^2$ are arbitrary.

In this \emph{circumvention dynamics}, the driver and evader rotates with the same angular velocities, which is proportional to the circumvention control $\kappa^c$. They also share the same center point, but have different radiuses and phases, so that the driver follows the evader from the outer circle orbit.

\subsubsection{Case 3 : The release mode}
The release mode is the case when the pursuit and circumvention stop, $\kappa^p \equiv 0$ and $\kappa^c \equiv 0$. Then, we have
\[ \ddot\bu_d(t) + \nu\dot\bu_d(t) = 0, \]
so that the driver stops its motion exponentially fast and tends to $\bar\bu_d(t) = \bu_d^*$.

 In turn, the relative position $\bu(t)$ follows
\[ \ddot\bu(t) + \nu\dot\bu(t) = f_e(|\bu(t)|)\bu(t).\]
Since $f_e$ is nonnegative, $\bu(t)$ grows toward its initial direction until $f_e(|\bu(t)|)$ is zero, under the potential energy $P_r(\bu)$:
\[ \ddot\bu(t) + \nu\dot\bu(t) = -\nabla P_r(\bu(t)),\quad P_r(\bu) := \int_{|\bu(0)|}^{|\bu(t)|} -rf_e(r)dr. \]

Therefore, in the \emph{release dynamics}, the evader will escape from the driver to a safe distance and flock together again. 

\subsubsection{The stability to the asymptotic motion}

Next, we want to prove that the dynamics asymptotically converges to the solutions \eqref{steady_0} and \eqref{steady_1} in the pursuit and circumvention mode, respectively.

Note that we assumed the constant controls, so that we may use the well-posedness result, Theorem \ref{T2.1}. Note again that we may assume $\kappa^p(t) \equiv 1$ from Remark \ref{R2.1}.
 
\begin{theorem}\label{T2.12}
Suppose that the function $f_d(r)$ and $f_e(r)$ satisfy \eqref{eq_f} and \eqref{f_g}.
Then, the following properties hold for constant controls $\kappa^p(t) \equiv 1$ and $\kappa^c(t) \equiv \kappa^c$.
\begin{itemize}
\item
If $\kappa^c = 0$, then $\bu(t)$ converges to a constant vector $\bu^* \in \mathbb R^2$ with $|\bu^*| = r_p$ and $f(r_p) = 0$, as in \eqref{eq_f}. Moreover, $\bu_d(t)$ and $\bu_e(t)$ converge asymptotically to the linear pursuit motion $\bar\bu_d(t)$ and $\bar\bu_e(t)$ in \eqref{steady_0}. The parameters $\phi_0$ and $\bu^*_e$ in \eqref{steady_1} are determined by the initial data.
\item
If $0<|\kappa^c| < \nu\sqrt{\gamma_m}$, then $|\bu(t)|$ converges to the constant distance $r_c$ with $f(r_c) = (\kappa^c)^2/\nu^2$, as in \eqref{r_c}. Moreover, $\bu_d(t)$ and $\bu_e(t)$ converges asymptotically to rotational circumvention motion $\bar\bu_d(t)$ and $\bar\bu_e(t)$ in \eqref{steady_1}. The parameters $\phi_1$ and $\bu^*_c$ in \eqref{steady_1} are determined by the initial data.
\end{itemize}
\end{theorem}

The proof of Theorem \ref{T2.12} uses LaSalle's invariance principle, which is described in Lemma \ref{L_nonlinear_zero} and \ref{L_evader_asymp}. 
The parameters $\phi_0$, $\bu_e^*$, $\phi_1$ and $\bu^*_c$ contain the information on the final direction and location of the driver and evader, where they are completely determined by the initial data. However, it is hard to specify these parameters explicitly, though the relative distances $r_p$ and $r_c$ are easily calculated by $f_d$ and $f_e$.

\subsection{Controllability on the evader's position}\label{sec:pre}

In the simplified model \eqref{GBR_simple}, Definition \ref{D1.1} can be described as follows. For a given $\bu_f \in \mathbb R^2$, we want to find control functions $\kappa^p$ and $\kappa^c$ which implies
\[ \bu_e(t_f) = \bu_f, \]
for some time $t_f>0$. 

As a heuristic guess based on Theorem \ref{T2.12}, we may use an off-bang-off control, which is also studied in \cite{escobedo2016optimal}:
\begin{equation}\label{control_kappa}
\kappa^p(t) = \kappa^p >0
\quad\text{and}\quad
\kappa^c(t) = \begin{cases} 0~ \quad \text{if }~ t \in [0,t_1)\cup(t_2,t_f],\\
\kappa^c \quad \text{if }~ t \in [t_1,t_2],
 \end{cases}
\end{equation}
for some constants $t_1$, $t_2$, $\kappa^p$ and $\kappa^c$.

Using \eqref{control_kappa}, we can concatenate two types of controlled trajectories, \eqref{steady_0} and \eqref{steady_1}. First, we start with the pursuit mode $\kappa^c(t) = 0$ and wait until time $t_1$, so that the driver is close enough to the evader. After $t_1$, we set a nonzero constant circumvention control $\kappa^c(t) = \kappa^c$, so that the driver and evader rotate continuously. After $t_2$, the evader tends to the linear motion again, and its direction is determined by the choice of $t_2$.

\begin{theorem}\label{T2.2}
Suppose that $f_d(r)$, $f_e(r)$, $\kappa^p$ and $\kappa^c$ satisfy \eqref{eq_f}, \eqref{f_g},
\[\kappa^p > 0\quad\text{and}\quad \kappa^c < \nu \sqrt{\kappa^p\gamma_m},\]
and the initial data satisfies $\bu^0_d \neq \bu^0_e$.
Then, for a given $\bu_f \in \mathbb R^2$, there exist $t_1$, $t_2$ and $t_f$ such that the solution of \eqref{GBR_simple} satisfies
\[ \bu_e(t_f) = \bu_f, \]
under the off-bang-off control \eqref{control_kappa}.
\end{theorem}

\begin{remark}
\begin{enumerate}
\item
Note that we only guarantee the position at $t_f$, and cannot say about the trajectory after $t_f$. The asymptotic solutions \eqref{steady_0} and \eqref{steady_1} have nonzero velocities, so that we can not get $\dot\bu_e(t_f) = 0$ in general (see also Figure \ref{fig:simus8-0} in the simulation section). 
\item
Since we assumed $f_e$ has an infinite integral in \eqref{f_g}, the solution from singular initial data $\bu^0_d = \bu^0_e$ is not well-defined. 
\item
The final time $t_f$ has a lower bound when the controls $\kappa^p(t)$ and $\kappa^c(t)$ are uniformly bounded by some constants. This is also from the asymptotic solutions, \eqref{steady_0} and \eqref{steady_1}, since they have bounded velocities.
\item
If we can use large enough controls, then we may expect $t_f$ can be arbitrary small, however, this cannot be proved rigorously in this paper since our analysis is only on the asymptotic motion.
\end{enumerate}
\end{remark}

\begin{figure}[ht]
  \centering
  {
    \includegraphics[width=0.7\textwidth]{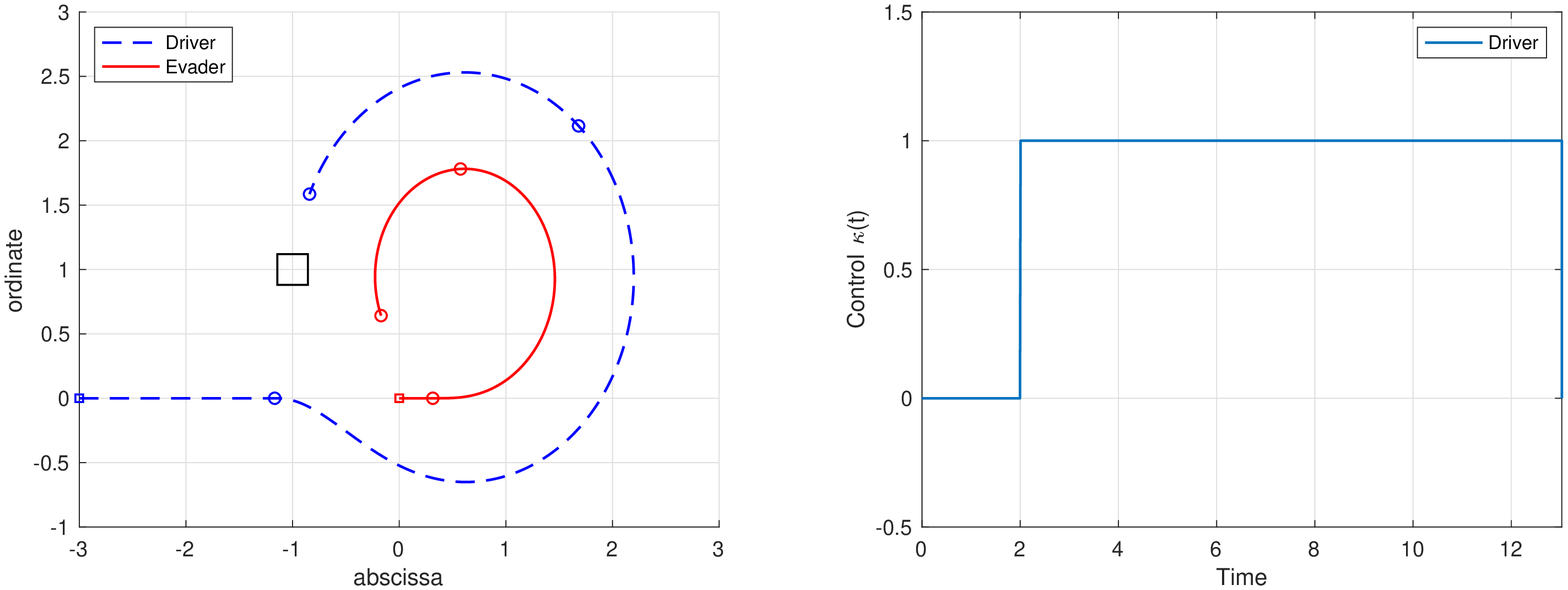}
  }
  \centering
  {
    \includegraphics[width=0.7\textwidth]{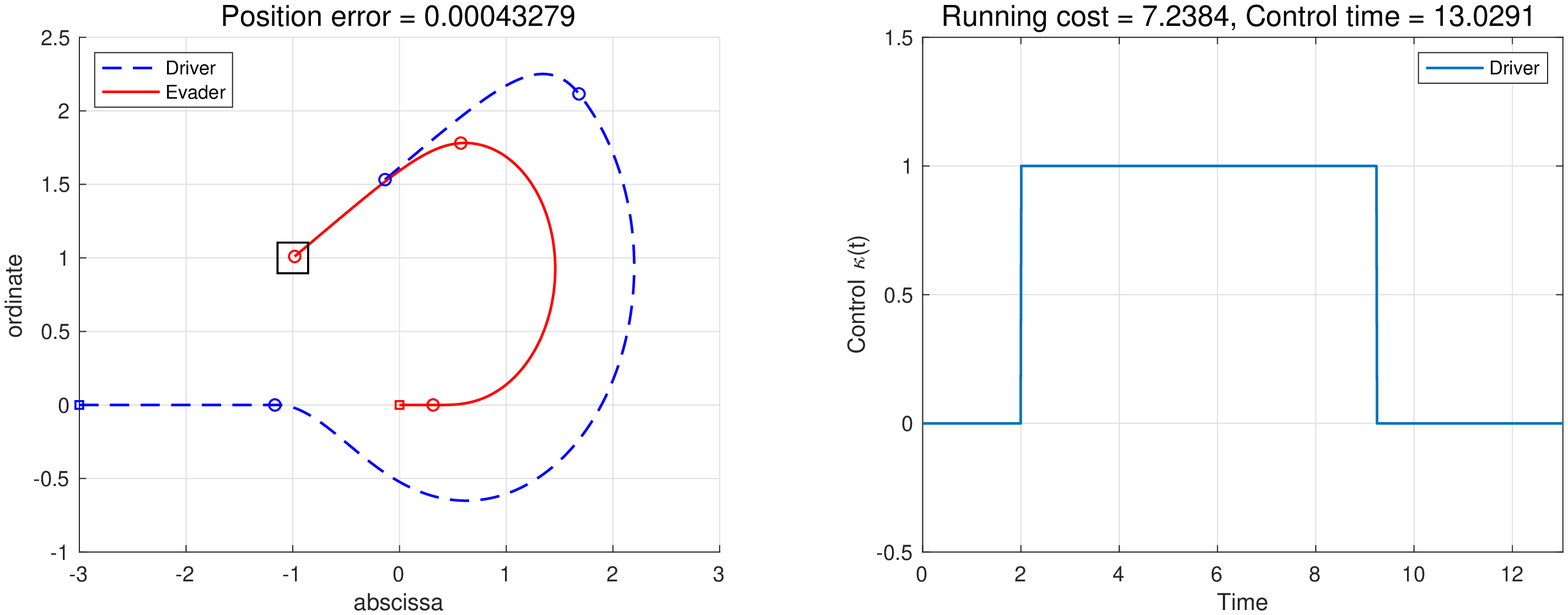}
  }
  \caption{ Diagrams for the off-bang-off control leading to $\bu_e(t_f) \simeq (-1,1)$: A constant control $\kappa^c(t) = 1$ after $t_1=2$ (top right) leads to rotational motion of two agents (top left). If we turn off the control after $t_2 \simeq 9.26$ (bottom right), then we can shot the evader to $(-1,1)$ (bottom left). Positions at time $t_1=2$, $t_2=9.256$ and $t_f=13.0421$ are marked as circles in the left figures. }
  \label{fig:simus1}
\end{figure}

Figure \ref{fig:simus1} shows one example with the interaction functions \eqref{fd_fe}. 
From Theorem \ref{T2.2}, we may achieve the final position exactly from any initial data.

In this simulation, the final position $(-1,1)$ is approximately achieved from given initial data $\bu_d = (-3,0)$ and $\bu_e = \bv_d = \bv_e = (0,0)$. We fix $t_1=2$ and $\kappa^c=1$ for convenience, and determine $t_2$ and $t_f$ using Matlab fmincon solver, minimizing the position error $|\bu_e(t_f) - (-1,1)|^2$.
Then, the driver perform circumvention maneuver for $t \in [2.0,9.256]$, and the evader passes the position near $(-1,1)$ at time $t_f = 13.0421$. This simulation is done with $1000$ time grids for $t \in [0,t_f]$.

The off-bang-off control \eqref{control_kappa} can be extended to multiple target points as in Figure \ref{fig:multi_target}. It describes one of the controlled trajectories passing through $6$ given points, $(3,3)$, $(4.5,5)$, $(6,1)$, $(9,3)$, $(7.5,5)$ and $(6,7)$, approximatly calculated by Matlab fmincon solver.

\begin{figure}[ht]
  \centering
  {
    \includegraphics[width=0.45\textwidth]{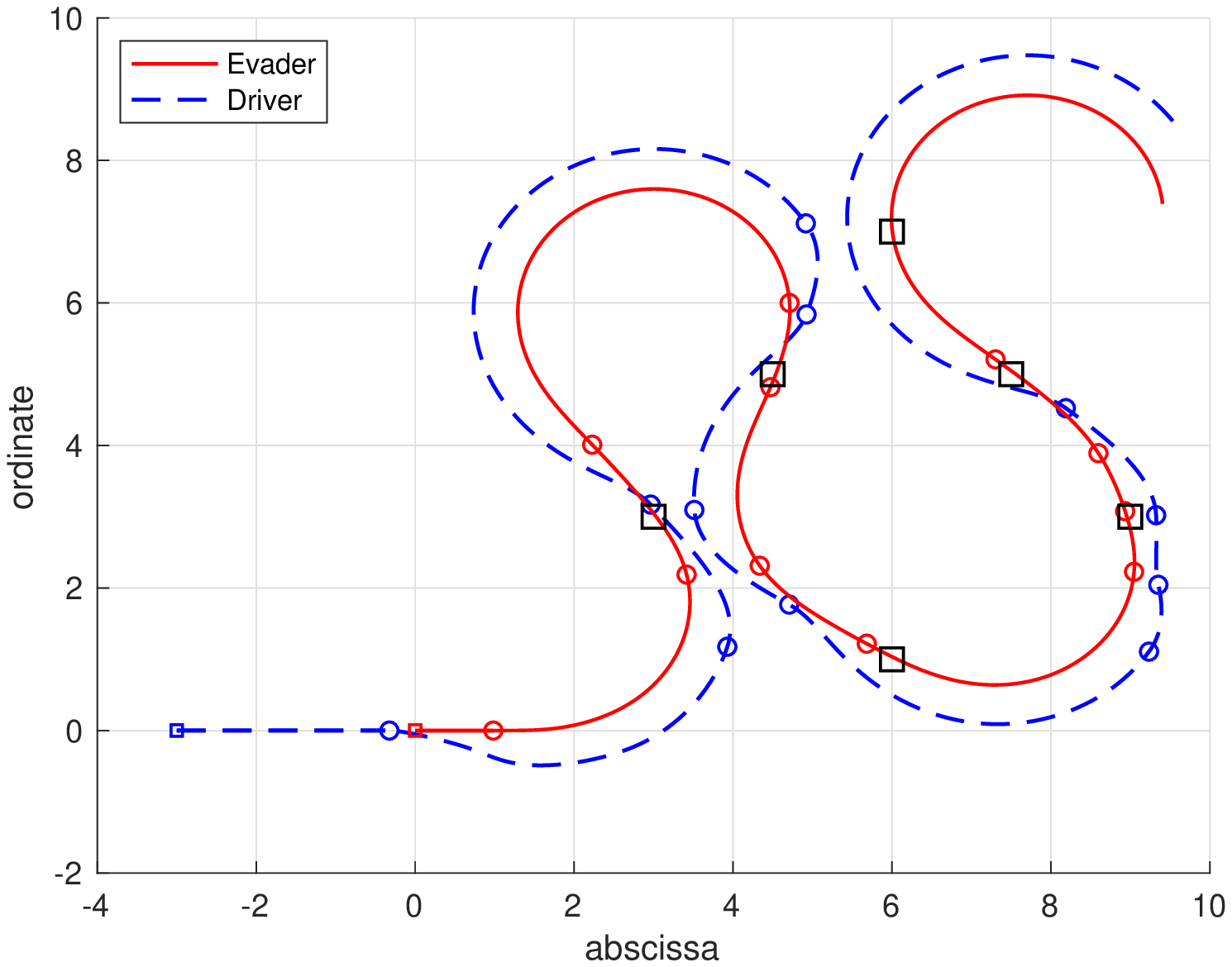}
    \includegraphics[width=0.45\textwidth]{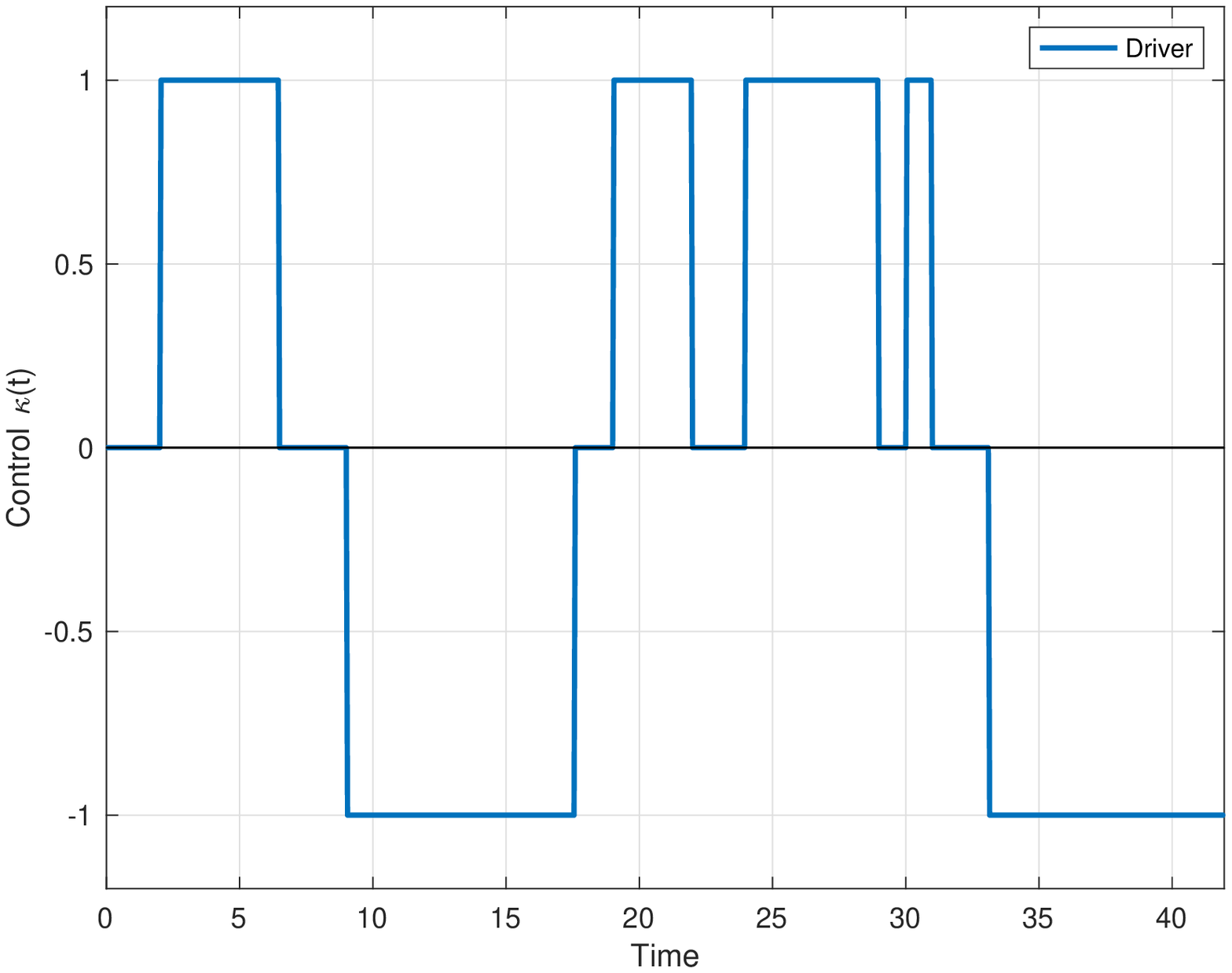}
  }
  \caption{A trajectory of the evader which passes near points $(3,3)$, $(4.5,5)$, $(6,1)$, $(9,3)$, $(7.5,5)$ and $(6,7)$ (left). The control $\kappa^c(t)$ is designed to minimize the distance between the trajectory and the target points one by one (right). Positions at times turning on and off the control are denoted by circle marks, and the target points are denoted by black boxes.}
  \label{fig:multi_target}
\end{figure}

From this simulation, we may observe that a constant circumvention control $\kappa^c(t) = 1$ make the evader rotates along an asymptotic circle, whose center and radius are completely determined by $t_1$ and $\kappa^c$. Since the driver will escape this circle after $t_2$, we may plot all possible final point $\bu_e(t_f)$ over the possible values of $t_2$ and $t_f$. Then, it will cover the whole space outside of the circle.

The proof of Theorem \ref{T2.2} follows the same argument, which is presented in Section \ref{sec:3.3}.

\vspace{1em}
\section{Proofs of results on the simplified model}\label{sec:control}

In this section, we present proofs of Theorem \ref{T2.1}, \ref{T2.12} and \ref{T2.2}. The proofs of the first two theorem use the Lyapunov function method, and then Theorem \ref{T2.2} is a consequence of two theorems.

All the Lyapunov functions are based on a standard energy $E$ of the relative position $\bu$. Let $\bx$ represent the phase state $\bx = (\bu_d,\bu_e,\bv_d,\bv_e)$, then the standard energy $E(\bx)$ is defined by
\begin{equation*}
E(\bx) := \frac{1}{2}|\bv|^2 + P(\bx) = \frac{1}{2}|\bv|^2 + \int_{r_p}^{|\bu|} sf(s)ds.
\end{equation*}
From the assumption \eqref{f_g} on $P$, we have
\[ E(\bx) = \infty \quad\text{if and only if}\quad |\bu|=0, \]
and $E(\bx)$ grows quadratically as $|\bu| \to \infty$. Hence, if every trajectory has a bounded energy along time, then we may conclude the upper and lower boundedness of $|\bu|$. 

\subsection{Global well-posedness}

From the definition of $E(\bx)$ and the equation \eqref{GBR_simple}, its time derivative of $E(\bx)$ can be calculated explicitly:
\begin{equation*}
\begin{aligned}
\dot E(\bx) &= \bv \cdot \dot {\bv} + f(|\bu|)\bu\cdot\dot{\bu}\\
&= \bv \cdot (-(\kappa^p(t)f_d(|\bu|)-f_e(|\bu|))\bu - \nu \bv + \kappa^c(t) \bu^\perp) + f(|\bu|)\bu \cdot \bv\\
&= -\nu|\bv|^2 + (1-\kappa^p(t))f_d(r)\bu\cdot\bv + \kappa^c(t) \bu^\perp \cdot \bv.
\end{aligned}
\end{equation*}

If $\kappa^p(t) \equiv 1$ and $\kappa^c(t) \equiv 0$, the pursuit mode, then we have $\dot E(\bx) \leq 0$. This guarantees that the energy $E(\bx)$ is uniformly bounded, hence, $|\bu(t)|$ is bounded from above and below. 

The problem is that, in general, $\dot E(\bx)$ may have a positive value and $E(\bx)$ can increase.
For the well-posedness of \eqref{GBR_simple}, we need to estimate the growth of $E(\bx)$ along time.

\begin{lemma}\label{L_nonlinear_collision}
Suppose that \eqref{eq_f} and \eqref{f_g} hold. If initially $\bu_d(0) \neq \bu_e(0)$ and $|\kappa^p(t)|$ and $|\kappa^c(t)|$ are uniformly bounded, then $\bu = \bu_d - \bu_e$ of the solution \eqref{GBR_simple} cannot blow-up to $\infty$ or hit $(0,0)$ in a finite time.

This guarantees the existence and uniqueness of the global solution $\bu_d(t)$ and $\bu_e(t)$ when $\bu_d(0) \neq \bu_e(0)$.
\end{lemma}

\begin{proof}
First, in order to prove that $\bu(t)$ does not blow-up in a finite time, we consider the time derivative of the standard energy:
\[ \dot E(\bx) = -\nu|\bv|^2 + (1-\kappa^p(t))f_d(r)\bu\cdot\bv + \kappa^c(t) \bu^\perp \cdot \bv. \]
From Young's inequality, we have
\begin{equation*}
\begin{aligned}
&|(1-\kappa^p(t))f_d(r)\bu\cdot\bv| \leq \frac{\nu}{2}\left(|\bv|^2 + \frac{(1-\kappa^p(t))^2|f_d(r)|^2}{\nu^2}|\bu|^2\right) \quad\text{and}\\
&|\kappa^c(t) \bu^\perp \cdot \bv| \leq \frac{\nu}{2}\left(|\bv|^2 + \frac{|\kappa^c(t)|^2}{\nu^2}|\bu|^2\right).
\end{aligned}
\end{equation*}
Hence, we get an estimate of the derivative $\dot E$:
\[ \dot E(\bx) \leq \frac{1}{2\nu}((1-\kappa^p(t))^2|f_d(r)|^2 + |\kappa^c(t)|^2)|\bu|^2. \]

From the assumption \eqref{f_g} on $f$, $E(\bx)$ grows at least quadratically as $|\bu| \to \infty$. Hence, from the uniform boundedness of $f_d(r)$, $\kappa^p(t)$ and $\kappa^c(t)$, there exists a constant $C$ such that
\[ \dot E(\bx) \leq CE(\bx)\quad \text{for large } |\bu|. \]
This implies that $E(\bx)$ is bounded for any finite $t$, and so is $\bu(t)$. 
\end{proof}

From Lemma \ref{L_nonlinear_collision}, the interaction function $f_e(r)$ is bounded along time if the controls are bounded, and then the equation \eqref{GBR_simple} is well-posed. The remaining part of Theorem \ref{T2.1} is to show that $\bu$ is uniformly bounded when $\kappa^p(t)$ is a nonzero constant, and $\kappa^c(t)$ is a constant.

In order to get a non-increasing energy function, we use hypocoercivity theory \cite{beauchard2011large,villani2009hypocoercivity}. Using this method, one can construct a decaying function by adding \emph{lower-order terms} (see \cite{villani2009hypocoercivity}) to the standard energy.
 Following the arguments of \cite{beauchard2011large}, the lower-order terms are given by the inner products of relative position and velocity, such as
\[ \nu|\bv|^2,\quad \nu \bu \cdot \bv,\quad \kappa \bu^\perp \cdot \bv \quad\text{and}\quad \kappa^2|\bu|^2. \]

\begin{lemma}\label{L_nonlinear_stable}
Suppose that \eqref{eq_f}, \eqref{f_g} hold and the controls are constant, 
\[ \kappa^p(t) \equiv \kappa^p>0,\quad \kappa^c(t) \equiv \kappa^c \quad\text{and}\quad |\kappa^c| < \nu \sqrt{\kappa^p\gamma_m}. \]
Then, the relative position $\bu$ of \eqref{GBR_simple} is uniformly bounded along time.
\end{lemma}

\begin{proof}
Without loss of generality, we assume $\kappa^p = 1$.
Define a Lyapunov function $L_\kappa(\bx)$,
\[ L_\kappa(\bx):= E(\bx) - \frac{\kappa^c}{\nu} \bu^\perp \cdot \bv, \]
then we have a nonpositive time derivative:
\begin{equation*}
\begin{aligned}
\dot L_\kappa(\bx) &= \bv \cdot \dot{\bv} + f(|\bu|) \bu \cdot \bv - \frac{\kappa^c}{\nu}\bu^\perp \cdot \dot{\bv}\\
&= \bv \cdot ( -f(|\bu|) \bu -\nu\bv + \kappa^c \bu^\perp ) + f(|\bu|) \bu \cdot \bv - \frac{\kappa^c}{\nu}\bu^\perp \cdot ( -f(|\bu|) \bu -\nu\bv + \kappa^c \bu^\perp )\\
&= -\nu|\bv|^2 + 2\kappa^c \bu^\perp \cdot \bv -\frac{(\kappa^c)^2}{\nu}|\bu|^2\\
&= -\nu \left| \bv - \frac{\kappa^c}{\nu}\bu^\perp \right|^2 \leq 0.
\end{aligned}
\end{equation*}

Next, we need to show that $L_{\kappa}(\bx) \to \infty$ when $|\bx| \to \infty$. Suppose that $f(r)$ satisfies \eqref{eq_f}, \eqref{f_g} and $|\kappa^c| < \nu \sqrt{\gamma_m}$. Then, from Young's inequality, there exists a small $\varepsilon >0$ satisfying
\[ \left|\frac{\kappa^c}{\nu}\bu^\perp \cdot \bv \right| \leq \frac{1}{2}(1-\varepsilon)|\bv|^2 + \frac{1}{2}(\gamma_m-2\varepsilon)|\bu|^2. \]
On the other hand, from the growth condition \eqref{f_g}, there exists a constant $M>0$ such that 
\[ E(\bx) = \frac{1}{2}|\bv|^2 + \int_{r_p}^{|\bu|} sf(s)ds \geq \frac{1}{2}|\bv|^2 + \frac{1}{2}(\gamma_m - \varepsilon)|\bu|^2\quad \text{for}~ |\bu|>M.\]
This implies that $L_\kappa(\bx)$ grows quadratically as $|\bx| \to \infty$:
\[ L_\kappa(\bx) = E(\bx) - \frac{\kappa^c}{\nu}\bu^\perp \cdot \bv \geq \frac{\varepsilon}{2}(|\bv|^2 + |\bu|^2)\quad \text{for}~ |\bu|>M.\]
Therefore, since $L_\kappa(\bx)$ is nonincreasing, $|\bu|^2+|\bv|^2$ is uniformly bounded from above, and $|\bu|^2$ is also uniformly bounded below.

\end{proof}

\subsection{Asymptotic behaviors}

From the boundedness of the relative dynamics, we may use LaSalle's invariance principle to show the asymptotic convergence. The following two lemmas complete the stability result, Theorem \ref{T2.12}. 
 
\begin{lemma}\label{L_nonlinear_zero}
Suppose that \eqref{eq_f} and \eqref{f_g} hold. Let $\bu_d(t)$ and $\bu_e(t)$ be the solution of \eqref{GBR_simple} with controls $\kappa^p(t) \equiv 1$ and $\kappa^c(t) \equiv 0$ from nonsingular initial data $\bu_d^0 \neq \bu_e^0$. Then, $\bu(t)$ converges to a constant vector $\bu_*$ with $|\bu_*| = r_p$ where $r_p$ is defined in \eqref{eq_f}.
 Moreover, $\bu_d(t)$ and $\bu_e(t)$ converge to the linear motion \eqref{steady_0}.
\end{lemma}
\begin{proof}
In Section \ref{sec:steady_0}, we observed that $\bu_d(t)$ and $\bu_e(t)$ converge to the linear motion \eqref{steady_0} if $\bu(t)$ is a constant vector $\bu_*$ with $|\bu_*| = r_p$.

Here, we need to prove that $\bu(t)$ converges to $\bu_*$ and the convergence is exponential. If these are true, then the $\bu$ terms in the equations of motion,
\begin{equation*}
\begin{aligned}
\dot {{\bv}}_d = -f_d(|\bu(t)|)\bu(t) - \nu \bv_d\quad\text{and}\quad
\dot {{\bv}}_e = -f_e(|\bu(t)|)\bu(t) - \nu \bv_e,
\end{aligned}
\end{equation*}
can be expressed $\bu_*$ and some error terms exponentially decaying to zero. This shows that $\bv_d(t)$ and $\bv_e(t)$ converges to $-f_d(r_p)\bu^*/\nu$ and $-f_e(r_p)\bu^*/\nu$, respectively. Hence, $\bu_d(t)$ and $\bu_e(t)$ converge to \eqref{steady_0}.

Since $\kappa^p(t) \equiv 1$ and $\kappa^c(t) \equiv 0$, we have $\dot E(\bx) = -\nu|\bv|^2 \leq 0$ and uniformly bounded $(\bu,\bv)$ along time (Lemma \ref{L_nonlinear_stable}). Hence, we may apply LaSalle's invariance principle: the $\omega$-limit set of $(\bu,\bv)$ is contained in the set 
\[ \{\bx ~:~ \dot E(\bx) = -\nu|\bv|^2 = 0\}.\]
Therefore, $\bu$ is globally attracted to some $\bu^* \in \mathbb R^2$ with $|\bu^*| = r_p$, which is the only solutions with $\bv(t) = 0$ in \eqref{GBR_rel}.

 Moreover, this convergence is exponential from \eqref{eq_f} since the potential $P(r)$ has a positive second derivative. In detail, the linearized equation of $\bu(t)$ around $\bar\bu(t)=\bu^*$ follows
 \[ \ddot{y} + \nu\dot{y} = -f'(r_p)(\bu^*\cdot y)\bu^* + o(\varepsilon), \]
for $\bu(t) = \bu^* + \varepsilon y$ with a small constant $\varepsilon \ll 1$. This shows that the convergence of $\bu(t)$ is exponential.
\end{proof}

\begin{lemma}\label{L_evader_asymp}
Suppose that \eqref{eq_f} and \eqref{f_g} hold and the controls $\kappa^p$ and $\kappa^c$ satisfy
\[ \kappa^p(t) \equiv 1 \quad\text{and}\quad \kappa^c(t) \equiv \kappa^c,\quad |\kappa^c| < \nu\sqrt{\gamma_m}. \]
Then, the solution $\bu$ of \eqref{GBR_simple} converges to a periodic motion with an angular velocity $\kappa^c/\nu$. Moreover, $\bu_e$ and $\bu_d$ asymptotically converge to \eqref{steady_1}.
\end{lemma}

\begin{proof}
As in the same argument of Lemma \ref{L_nonlinear_zero}, we need to prove that $\bu$ converges exponentially to the periodic motion in Section \ref{sec:steady_1}. 

Lemma \ref{L_nonlinear_stable} shows that $\bu$ asymptotically approaches to the set
\[ \{ \bx ~|~ \dot L_\kappa(\bx) = 0 \}. \]
Then, from the LaSalle's principle, we need to find an invariant set satisfying $\nu \bv = \kappa^c \bu^\perp$. This corresponds to a circular motion around zero, hence, $\bu(t)$ converges to $\bar \bu(t)$ of \eqref{steady_1u}.

In turn, $\bv_d$ and $\bv_e$ converge to the solution of
\begin{equation*}
\begin{aligned}
\dot {{\bv}}_d = -f_d(r_c)\bar \bu + \kappa^c g(r_c) \bar \bu^\perp - \nu \bv_d\quad\text{and}\quad
\dot {{\bv}}_e = -f_d(r_c)\bar \bu - \nu \bv_e,
\end{aligned}
\end{equation*}
in the same argument as in Lemma \ref{L_nonlinear_zero}.
These are frictional motions with periodic external forces, which are discussed in Section \ref{sec:steady_1}. Hence, $\bar \bu_d$ and $\bar \bu_e$ eventually converge to the rotational motion of \eqref{steady_1}.

\end{proof}

\subsection{Controllability of the evader's position}\label{sec:3.3}

The idea of Theorem \ref{T2.2} is to use the asymptotic solutions \eqref{steady_0} and \eqref{steady_1} as in Figure \ref{fig:simus1}. 
This requires long enough final time $t_f$ since we need to wait after each change of $\kappa^c(t)$ to make the solution close to its asymptotic motion.

\begin{proof}[Proof of Theorem \ref{T2.2}]
We use Theorem \ref{T2.12} to guarantee that the solutions are bounded and converge to \eqref{steady_0} or \eqref{steady_1}.

 For any constant $\kappa^c$ satisfying $|\kappa^c| < \nu \sqrt{\gamma_m}$, we consider the set of off-bang-off controls
\begin{equation}\label{eq_kappa}
\kappa^p(t) = 1
\quad\text{and}\quad
\kappa^c(t) = \begin{cases} 0 ~\quad \text{if }~ t \in [0,t_1)\cup(t_2,t_f],\\
\kappa^c \quad \text{if }~ t \in [t_1,t_2], \end{cases}
\end{equation}
over the constants $t_1$, $t_2$, and $t_f$.
 Then, we may define the reachable set $\mathcal R$ over such controls,
\[ \mathcal R := \bigcup_{t_1,t_2,t_f} \mathcal R(t_1,t_2,t_f), \]
where $\mathcal R(t_1,t_2,t_f)$ is the final position $\{\bu_e(t_f)\}$ of the evader for \eqref{eq_kappa}. If $\mathcal R $ covers the whole space $\mathbb R^2$, then the evader's position is controllable.

First, we may fix $t_1$ and consider large enough $t_2$. Then, from Theorem \ref{T2.12}, the position of evader $\bu_e$ converges to a periodic orbit in \eqref{steady_1}. After turning off the control at $t_2$, $\bu_e(t)$ will tends to a uniform linear motion and goes to the infinity point. Note that the rotational component decays exponentially in the sense that
\[ \frac{d}{dt}(\bu^\perp\cdot\bv) = -\nu\bu^\perp\cdot\bv. \]
This implies that the trajectory of $\bu_e(t)$ draws a curve from $\bu_e(t_2)$ to $\infty$ with a bounded rotational angle, and finally tends to a straight line. Therefore, the union of the reachable set over $t_2 \in [t_1,\infty)$ and $t_f \in [t_2,\infty)$ covers the whole area outside of the stable orbit,
\[ \mathbb R^2 - \overline{B_{r_e}(\bu^*)} \subset \bigcup_{t_2,t_f} \mathcal R(t_1,t_2,t_f), \]
where $\overline{B_{r_e}(\bu^*)}$ is the closed ball with the center $\bu^*$ and radius $r_e$ from \eqref{steady_1}. 

Finally, note that the choice of $t_1$ determines the location $\bu^*$ of the stable orbit. Therefore, the reachable set $\mathcal R$ covers $\mathbb R^2$. 

\end{proof}

\vspace{1em}
\section{Numerical simulations on the guidance-by-repulsion model}
\label{sec:multi}

In this section, we simulate optimal control strategies on the guidance-by-repulsion model \eqref{GBR_general}. We perform numerical simulations with various initial positions and control costs, including the case with many drivers or many evaders. This is a purely computational study, however, it reveals key characters and patterns of the optimal controls.

In Section \ref{sec:sim_driving}-\ref{sec:Nsimul}, we study optimal control problems compared to the off-bang-off control, in terms of the running cost and the control time. In the simplified model \eqref{GBR_simple}, we already observed in Theorem \ref{T2.2} that there exists an off-bang-off control \eqref{control_kappa} which leads to the target $\bu_e(t_f) = \bu_f$. This corresponds to Definition \ref{D1.1} and the studies in \cite{king2012selfish,strombom2014solving} that we want to gather the sheep into a desired area such as a sheepcote.

Next, in Section \ref{sec:sim_herding}, we consider the stabilization on the final position of the herd, which is similar to \cite{burger2016controlling,lien2004shepherding}. The drivers wants to capture the evaders in a small region, so they will occupy a formation around the target point.

Here the nonlinearity of the model is the same as in \eqref{fd_fe},
\begin{equation*}
\begin{aligned}
f_d(r) = 1,\quad f_e(r) = \frac{1}{r^2},
\quad \psi_d(r) = \frac{1}{2r^4} \quad\text{and}\quad \psi_e(r) = 10\left(\frac{(0.1)^2}{r^2}-\frac{(0.1)^4}{r^4}\right),
\end{aligned}
\end{equation*}
where we use bounded pursuit and circumvention controls,
\[ 0 \leq \kappa_j^p(t) \leq 1 \quad\text{and}\quad -5 \leq \kappa_j^c(t) \leq 5,\quad\text{for }~ j=1,\cdots,M. \]

\subsection{The formulation of the optimal control problem}\label{sec:sim_driving}

In this part, we explain the formulation of the optimal control problem in the guidance-by-repulsion model. This optimal control problem is not an obvious task since the model is sensitive to the control functions and the final time.

We use the gradient descent method to optimize control functions, where the gradients are calculated by the adjoint system. 
Each simulation in this section usually takes several minutes, sometimes several tens of minutes on a typical laptop.

\subsubsection{Two cost functions for optimization}
Since the objective of the control is to get $\bu_{ei}(t_f)$ close to $\bu_f$, we consider the following three components of the cost functions:
\begin{itemize}
\item
\emph{The position error}: The distances between the final positions of the evaders $\bu_{ei}(t_f)$ and the target point $\bu_f$:
\[ \frac{1}{N} \sum_{i=1}^N |\bu_{ei}(t_f) - \bu_f|^2. \]
\item
\emph{The running cost}: The $L^2$-norms of the control functions $\kappa^c_j(t)$:
\[ \frac{1}{M} \sum_{j=1}^M \int_0^{t_f} |\kappa_j^c(t)|^2 dt. \]
 
\item
\emph{The control time}: The time taken to drive the evaders:
\[ t_f. \]
\end{itemize}
We define the total cost function $J$ with these three components,
\begin{equation}\label{OCP0} 
J = \frac{1}{N} \sum_{i=1}^N |\bu_{ei}(t_f) - \bu_f|^2 + \frac{\delta_1}{M} \sum_{j=1}^M \int_0^{t_f} |\kappa_j^c(t)|^2 dt + \delta_2 t_f,
\end{equation}
where $\delta_1$ and $\delta_2$ are regularization constants. The use of running cost $|\kappa^c_j(t)|^2$ suggests that the pursuit mode is considered to be a resting state while the circumvention maneuver consumes costs for rotations. 

In this setting, we will use two types of cost functions, $(\delta_1,\delta_2) = (0.001,0)$ and $(0.001,0.01)$. Along with their goals of the optimization, we may call the optimal solutions of them as the strategies \emph{minimizing the running cost}, and the strategies \emph{minimizing the control time}, respectively.

\subsubsection{The flexible final time}
The flexibility of the final time $t_f$ is important since Definition \ref{D1.1} does not require the final velocities to be zero.
The simulation with a fixed final time may lead to an unreasonable control functions since the asymptotic solutions \eqref{steady_0} and \eqref{steady_1} are not steady and continuously moving.

For example, in the simplified model \eqref{GBR_simple}, suppose that the initial positions are given by
 \[ \bu_d(0)=(-1,0),\quad\bu_e(0)=(0,0)\quad\text{and}\quad\bu_f=(1,0), \]
where the initial velocities are zero.
Note that the velocity of the evader is completely determined by the distances between the driver and evader. Hence, there exists an upper bound of the final time to drive it through the $x$-axis, less than $f_e(2)^{-1}$. This implies that we will get a strange trajectory with the final time larger than $f_e(2)^{-1}$. 

 In order to implement a flexible final time, we use a time-scaling map $T:[0,1] \to [0,t_f]$ with a positive bounded velocity $0<C_1<T'(s)<C_2$ for some constants $C_1$ and $C_2$. For example, suppose that we want to deal with the following dynamics
 \[ \ddot \bu_d + \nu\dot \bu_d =\kappa^p(t)f_d(|\bu|)\bu +\kappa^c(t)\bu^\perp,\quad t \in [0,t_f]. \]
Then, we interpret this equation into a new time variable $s \in [0,1]$ using $t = T(s)$,
\[ \frac{1}{T'(s)^2}\frac{d^2}{ds^2}\bu_d + \frac{1}{T'(s)}\nu\frac{d}{ds}\bu_d = \kappa^p(T(s))f_d(|\bu(T(s))|)\bu(T(s)) + \kappa^c(T(s))\bu(T(s))^\perp,\quad s\in[0,1]. \]
Therefore, the cost function $J$ is transformed to $\bar J(\bar \kappa^p(s),\bar \kappa^c(s),T(s))$, where we can find a minimizer $(\bar \kappa^p(s),\bar \kappa^c(s),T(s))$ with a fixed time interval $[0,1]$. Then, the optimal controls of the original system will be 
\[ \kappa^p(t) = \bar\kappa^p(T^{-1}(s))\quad\text{and}\quad \kappa^c(t) = \bar\kappa^c(T^{-1}(s)). \]

\subsubsection{The controllability and the initial guess}

We want to set a reasonable initial guess on the control functions since  the model \eqref{GBR_simple} is sensitive to the controls. For example, from Figure \ref{fig:simus1} with one driver and one evader, there may exist a control with one more rotation of the evader, which might be a local minimizer of the cost \eqref{OCP0}.
Hence, we need a proper initial guess in order to boost the optimization algorithm.

For the case of one driver and one evader, we can start with the off-bang-off controls in Theorem \ref{T2.2}. Among them, the constant controls are one of the special cases we may try first. The reachable set over the constant controls can not cover the whole domain, however, it works for a properly given target point, especially when the target is far from the initial positions.

While the controllability results holds in the simplified model \eqref{GBR_simple}, the concept of Definition \ref{D1.1} is difficult to be achieved by hand. For example, if we consider the off-bang-off controls, then we need to check all the evaders are in a desired area and did not escape away from the drivers.
For the problem with many drivers or many evaders, as an initial guess, we consider a piecewise constant control functions given by hand.

\subsection{Optimal control strategies with one driver and one evader}\label{sec:11simul}

Figure \ref{fig:simus2} shows a simulation with constant pursuit and circumvention controls, where we used the same initial data and target point as in Figure \ref{fig:simus1}.
This constant control is a reference solution we use as an initial guess, where it is calculated by Matlab fmincon solver.

\begin{figure}[ht]
  \centering
  {
    \includegraphics[width=0.7\textwidth]{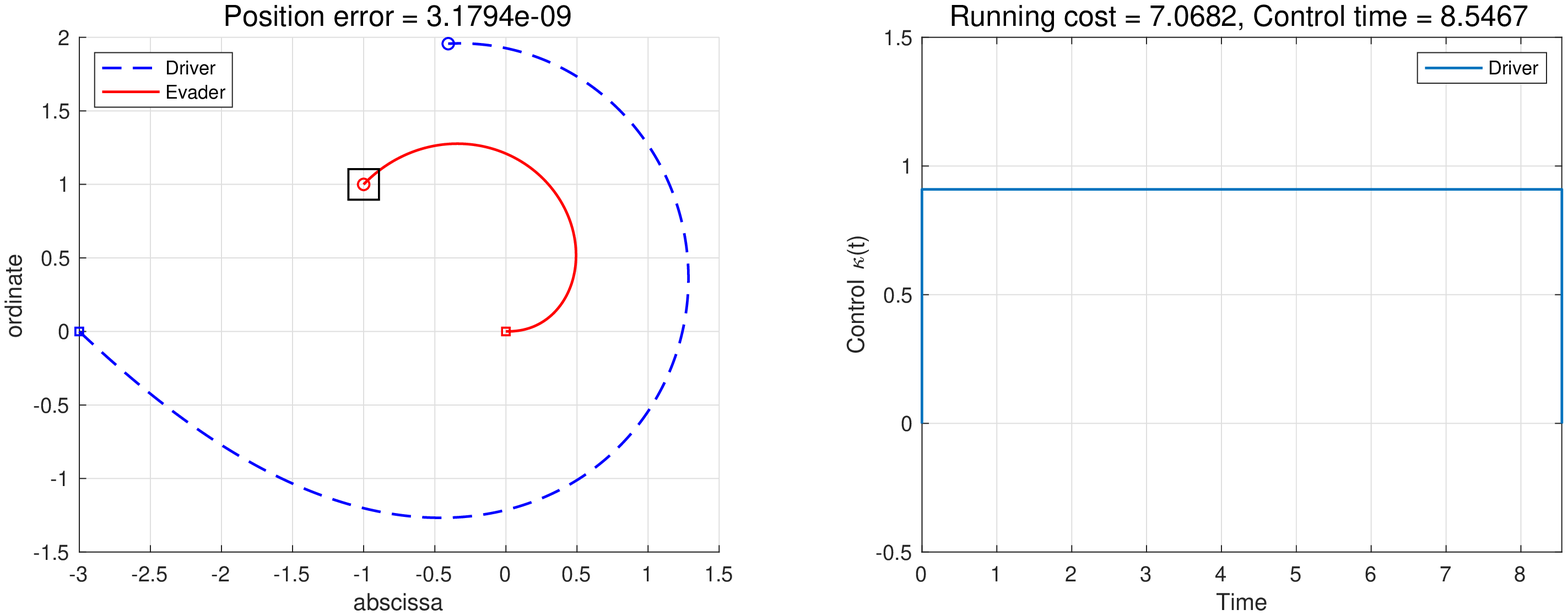}
  }
  \caption{ Diagrams for the constant control leading to $\bu_e(t_f) \simeq (-1,1)$ with $\kappa^p(t) \equiv 1$, $\kappa^c(t) \equiv 1.5662$ and $t_f = 5.1727$.}
  \label{fig:simus2}
\end{figure}

\subsubsection{Strategies minimizing running cost}

We first use the cost function \eqref{OCP0} with $(\delta_1,\delta_2) = (0.001,0)$:
\begin{equation}\label{OCP1}
J(\kappa^c(\cdot)) := |\bu_e(t_f) - \bu_f|^2 + \delta_1 \int_0^{t_f} |\kappa^c(t)|^2dt,\quad \delta_1 = 0.001,
\end{equation}
which deals with the running costs of the circumvention controls. Note that if the final error can be reduced to the order of $\delta_1$, then the simulation with \eqref{OCP1} produces optimal strategies which mainly minimizes $\|\kappa^c(\cdot)\|_{L^2(0,t_f)}$.

\begin{figure}[ht]
  \centering
  {
    \includegraphics[width=0.7\textwidth]{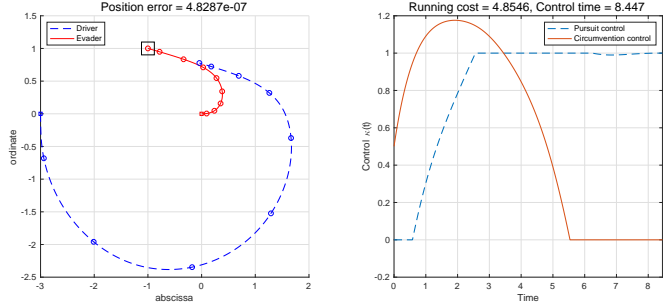}
  }
  \caption{Diagrams for the optimal control leading to $\bu_e(t_f) \simeq (-1,1)$ with the cost \eqref{OCP1}. The positions at the integer times and the final time are marked by circles.}
  \label{fig:simus3}
\end{figure}

Figure \ref{fig:simus3} shows the optimal control based on \eqref{OCP1}. The driver performs the circumvention motion in the beginning, and follows the pursuit dynamics near the final time. The circumvention control $\kappa^c(t)$ tends to zero after some time while the pursuit control $\kappa^p(t)$ goes to $1$.
This phenomena can be observed from various initial and final positions with zero initial velocities. 

Compared to the constant control in Figure \ref{fig:simus2}, the optimal control in Figure \ref{fig:simus3} reduces the running cost nearly one third (from $7.0682$ to $4.8546$) with a similar control time ($8.4470$, compared to the constant control case, $8.5467$). Note that it starts with the zero pursuit control to increase the effect of circumvention. This action reduces the running cost while the control time may increase.

 The position error is larger than the constant control, but still in an acceptable value with respect to the running cost. If we choose smaller value of $\delta_1$, we may reduce the position error to an arbitrarily small value without a big change of the running cost.

\subsubsection{Strategies minimizing control time}

Compared to the off-bang-off strategy in Figure \ref{fig:simus1}, it is natural to ask the optimal control minimizing the driving time. Since we assumed that the pursuit control is bounded, $0 \leq \kappa^p(t)\leq 1$, the speed of the driver is bounded so that there exists the minimum control time. For this, we use the cost function \eqref{OCP0} with the control time:
\begin{equation}\label{OCP2}
J(\kappa^c(\cdot)) := |\bu_e(t_f) - \bu_f|^2 + \delta_1 \int_0^{t_f} |\kappa^c(t)|^2dt + \delta_2 t_f,\quad (\delta_1,\delta_2) = (0.001,0.01).
\end{equation}

\begin{figure}[ht]
  \centering
  {
    \includegraphics[width=0.7\textwidth]{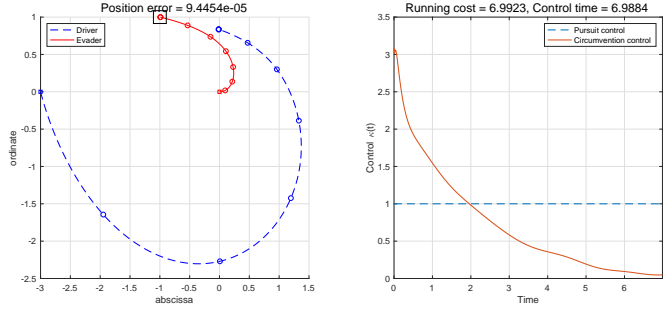}
  }
  \caption{Diagrams for the optimal control leading to $\bu_e(t_f) \simeq (-1,1)$ with the cost \eqref{OCP2}. The positions at the integer times and the final time are marked by circles.}
  \label{fig:simus4}
\end{figure}

Figure \ref{fig:simus4} describes the optimal trajectories with respect to \eqref{OCP2}, which has more attention to the control time. Compared to Figure \ref{fig:simus3}, the circumvention control function wants to have a stronger value in the beginning and decays fast to zero. Moreover, the pursuit control is constantly $1$ to minimize the control time.

 As a result, it reduces the control time from $8.4470$ to $6.9884$, nearly $13\%$ of the time. Meanwhile, the running cost changed to $6.9923$ from $4.8546$, which is quite significant though it is still less than the constant control, $7.0682$. 

 This strategy is reasonable in terms of the evader's speed, since the evader is faster when the driver and evader are closer (note that $r_p < r_c$ from \eqref{eq_f} and \eqref{r_c}). To occupy enough time near the pursuit dynamics, we can observe that the driver moves fast initially to place itself behind the target point. Hence, the amplitude and change of $\kappa^c(t)$ are much significant than in Figure \ref{fig:simus3}.

In summary, the optimal control functions share a basic strategy that the driver first rotates to get a right direction and then drive the evader with small values of $\kappa^c(t)$. The pursuit control can be a small value in the rotation process, but most of time it takes the maximal value $1$ to reduce the control time.

\subsection{Control strategies with many drivers or many evaders}\label{sec:Nsimul}

In this part, we simulate optimal control trajectories with more than one driver or one evader. In order to compare with the simulations in Section \ref{sec:11simul}, we focus on the case that the evaders are initially gathered and the drivers are ourside of the herd.

\subsubsection{Controls with two drivers and one evader}
We start with two drivers and one evader, where two drivers cooperate to steer the evader. 
Figure \ref{fig:simus5} shows two simulations with similar initial data and cost functions to Figure \ref{fig:simus3} and \ref{fig:simus4}. The drivers start with the initial positions $(-3,0.5)$ and $(-3,-0.5)$.

\begin{figure}[ht]
  \centering
  {
    \includegraphics[width=0.7\textwidth]{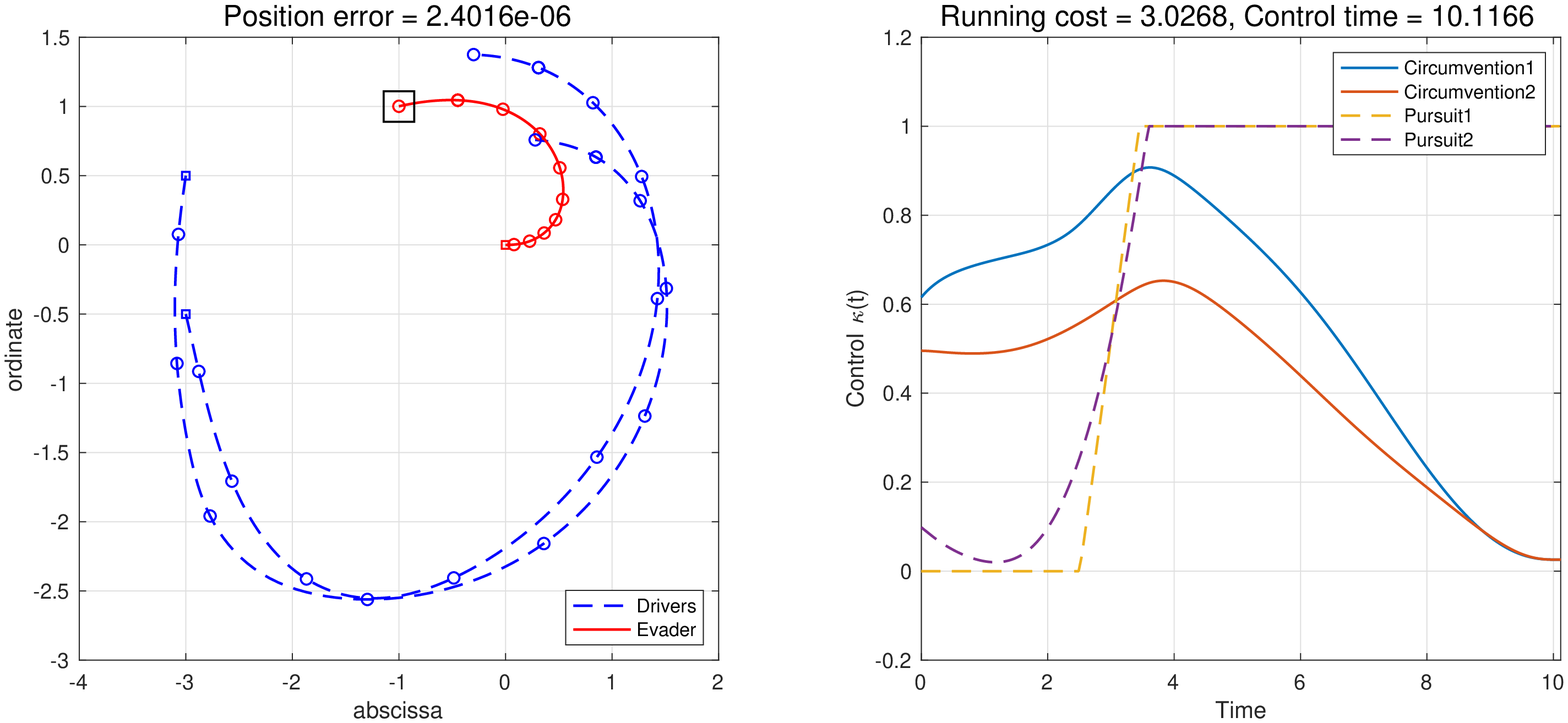}
  }
  \centering
  {
    \includegraphics[width=0.7\textwidth]{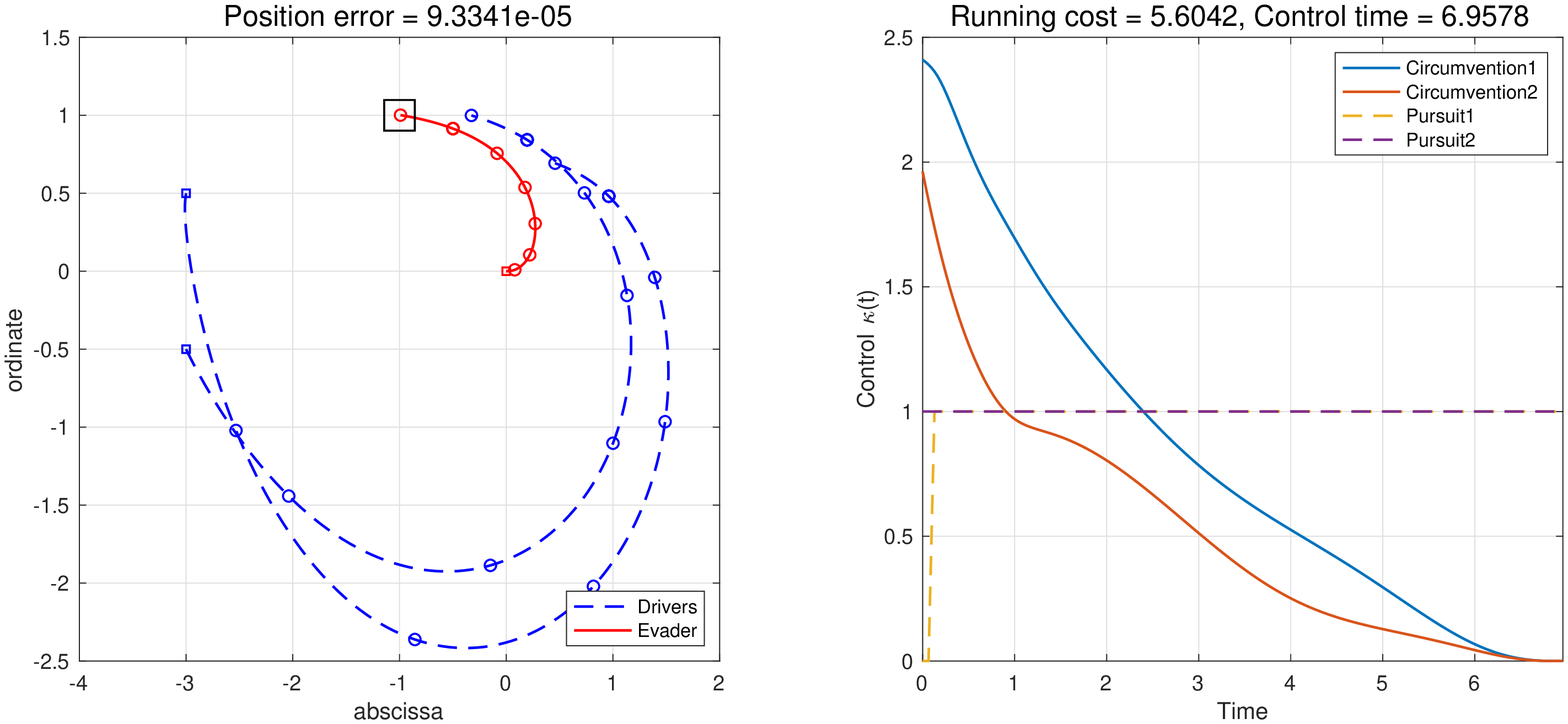}
  }
  \caption{Diagrams for the control leading to $\bu_{e1}(t_f) \simeq (-1,1)$ with two drivers. The above graph is for the minimum running cost, and the below one is for the minimal time in \eqref{OCP0} with $(\delta_1,\delta_2) = (0.001,0)$ and $(0.001,0.01)$, respectively.}
  \label{fig:simus5}
\end{figure}

For an efficient control, both drivers follow similar trajectories with the one-driver case, except for the separation of two drivers caused by the interaction $\psi_d$. Note that the pursuit and circumvention controls follow similar values to the simulations in Figure \ref{fig:simus3} and \ref{fig:simus4}. 

In order to minimize the control time, the drivers need strong circumvention control initially in the lower graph of Figure \ref{fig:simus5}. In this way, it minimizes the control time ($6.9578$) nearly one third, compared to the solution minimizing the running cost ($10.1166$). Instead, the running cost increases to $5.6042$ from $3.0268$ to rotate effectively in a short time. We can also observe that the circumvention controls are nearly zero in the final time. 

The control time is also reduced compared to the one-driver cases. This is more significant for the optimal strategy minimizing control time, since one of the drivers can approach the evader more closer.

On the other hand, Figure \ref{fig:simus6} starts with separated initial positions,
\[ \bu_{d1}(0) = (-2,-2),\quad \bu_{d2}(0) = (2,-2)\quad \text{and}\quad \bu_{e1} = (0,0),\]
while the desired target is $(-3,0)$, the opposite direction.
If initially the evader does not move toward the desired direction, in the one-driver case, the driver should perform a circumvention maneuver to rotate the evaders. However, in the multi-driver situation, they can cooperate to correct the direction.

\begin{figure}[]
  \centering
  {
    \includegraphics[width=0.7\textwidth]{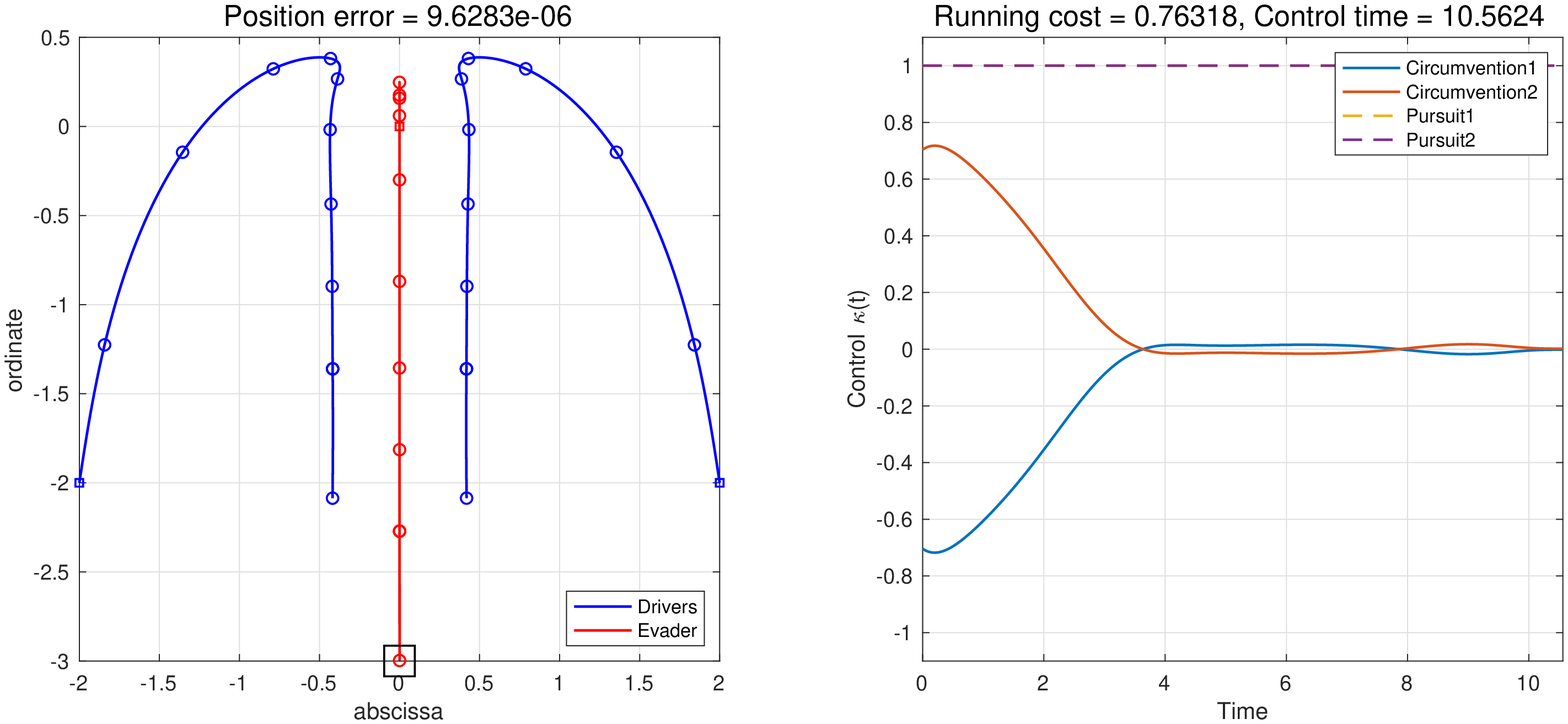}
  }
  \centering
  {
    \includegraphics[width=0.7\textwidth]{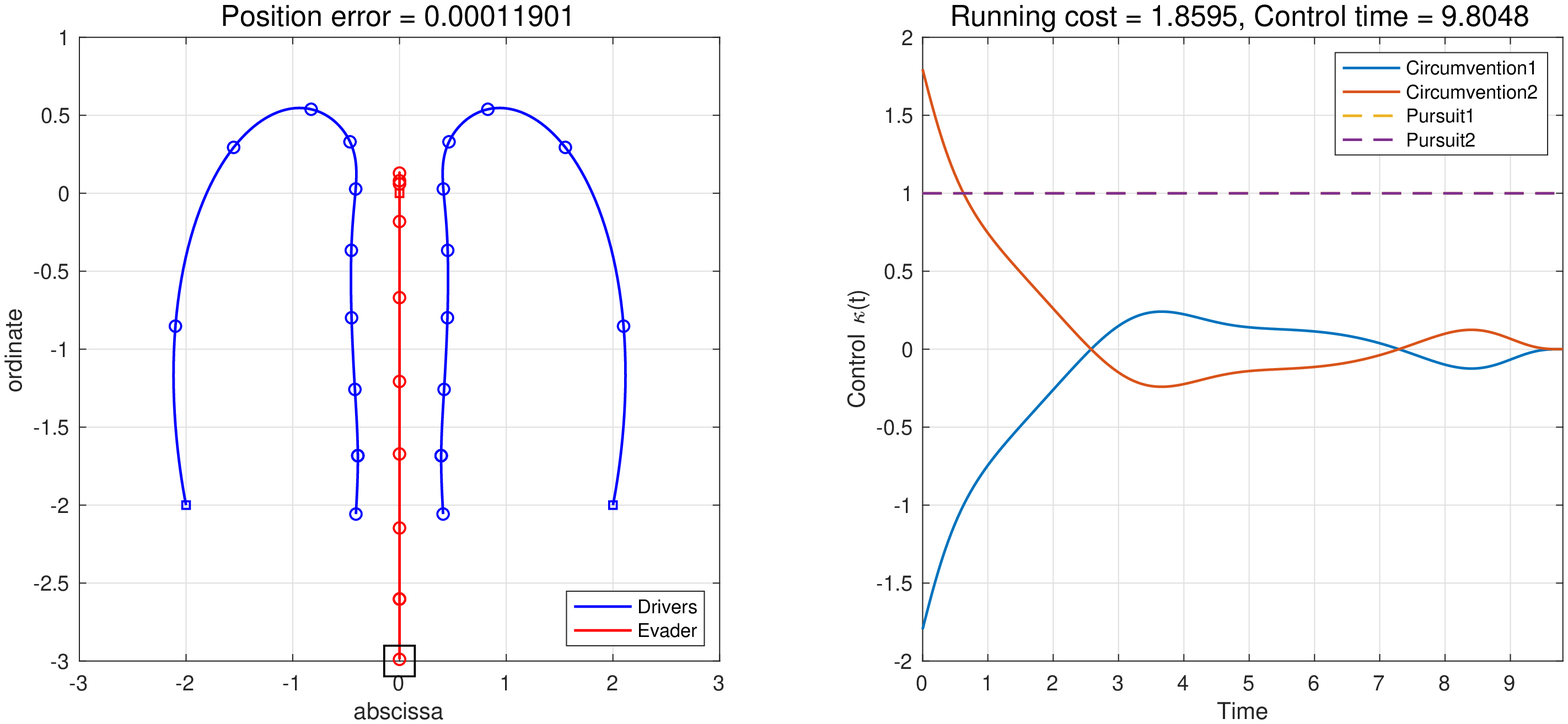}
  }
  \caption{Diagrams for the control leading to $\bu_{e1}(t_f) \simeq (0,-3)$ with two drivers. The above graph is for the minimum running cost, and the below one is for the minimal time in \eqref{OCP0} with $(\delta_1,\delta_2) = (0.001,0)$ and $(0.001,0.01)$, respectively.}
  \label{fig:simus6}
\end{figure}

In Figure \ref{fig:simus6}, the pursuit controls are all $1$, and the circumvention controls are symmetric to keep the position of the evader in the $y$-axis. We can observe a similar pattern here, except for the nonzero circumvention controls near the final time. In this way, the drivers are more closer to the evader so that the evader moves faster.

\subsubsection{Controls with two drivers and two evaders}

Similar strategies also appear in the multi-evader case, Figure \ref{fig:simus7}. This simulation starts with initial data
\[ \bu_{d1}(0) = (-4,3),\quad \bu_{d2}(0) = (-4,-3),\quad \bu_{e1}(0) = (-2,0.5),\quad \text{and}\quad \bu_{e1} = (-2,-0.5),\]
with zero initial velocities, and the desired target point $\bu_f$ is $(0,0)$.

 Since the final positions of the evaders cannot coincide with the target point $(0,0)$ exactly, the position error is bigger than the one-evader case. In order to reduce the restriction on the final position, we use different regularization constants, $(\delta_1,\delta_2) = (0.01,0)$ for minimizing running cost and $(\delta_1,\delta_2) = (0.01,0.1)$ for minimizing control time.

\begin{figure}[]
  \centering
  {
    \includegraphics[width=0.7\textwidth]{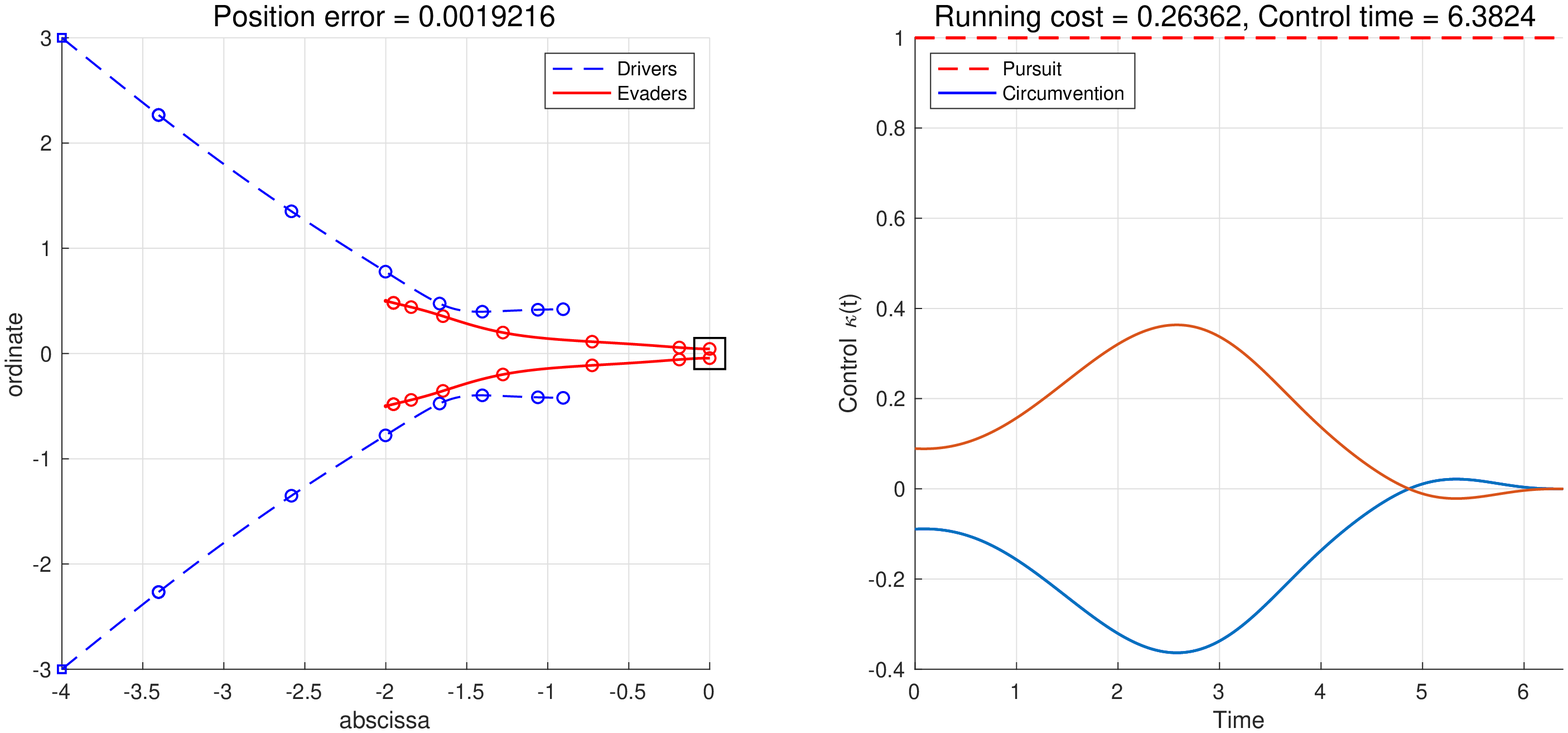}
  }
  \centering
  {
    \includegraphics[width=0.7\textwidth]{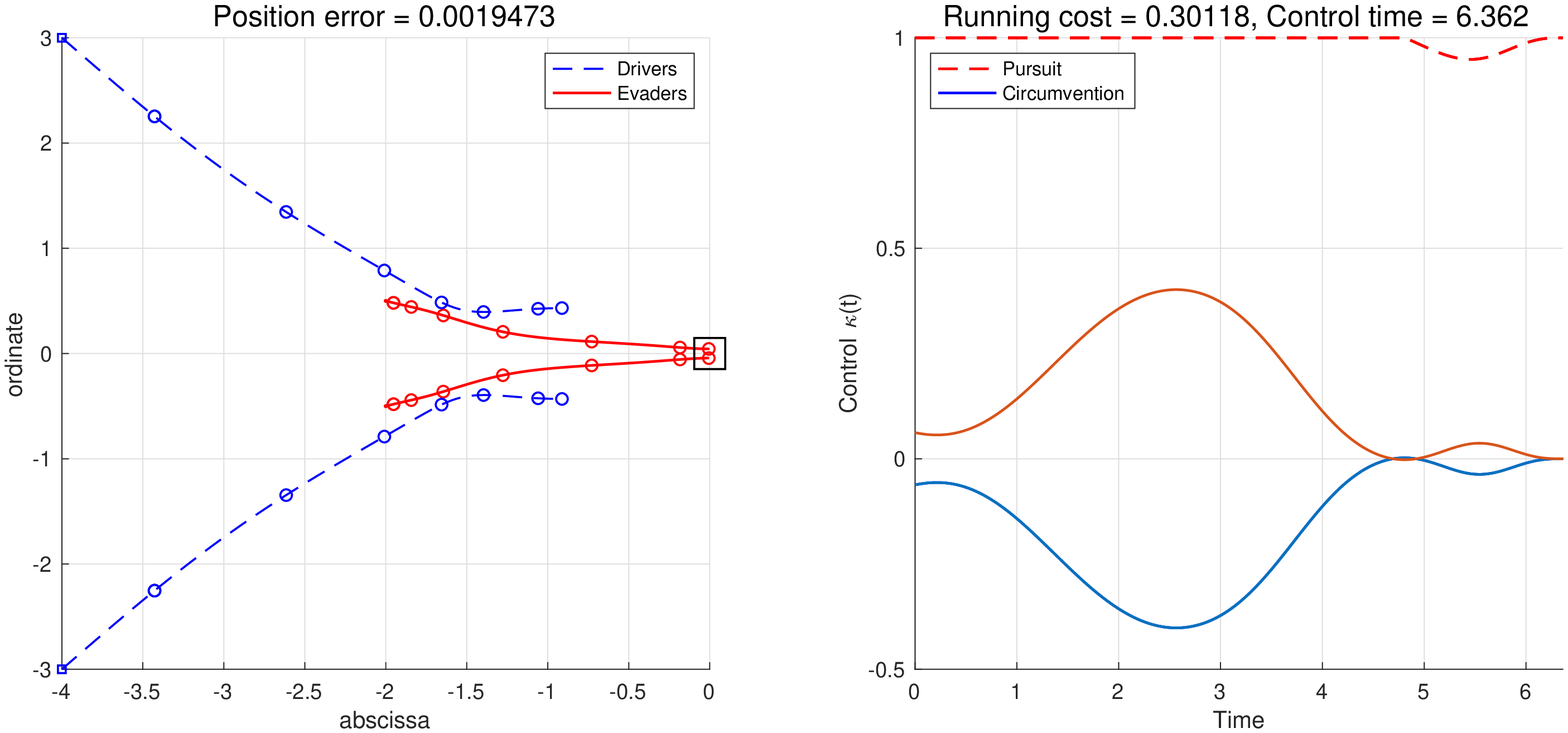}
  }
  \caption{Diagrams for the control leading to $\bu_{e1}(t_f), \bu_{e2}(t_f) \simeq (2,0)$ with two drivers. The above graph is for minimal running cost $(\delta_1,\delta_2) = (0.01,0)$, and the below one is for minimal time $(\delta_1,\delta_2) = (0.01,0.1)$ from the cost functions of \eqref{OCP0}.}
  \label{fig:simus7}
\end{figure}

If there are more than two evaders, the gathering of evaders is not a trivial task. In Figure \ref{fig:simus7}, the circumvention controls are nonzero when they pursuit the evader, since it may lead the evaders to separate in a vertical direction if the drivers are too close.

In this simulation, the minimal time strategy is not much different from the strategy minimizing the running cost, which is reasonable since they start with well-ordered initial positions in the horizontal direction. Near the final time, $t \geq 4.5$, there are slight change of control functions in the same philosophy as in Figure \ref{fig:simus6}; the drivers approach to the evaders a little more, with a slight increase of the position error and the running cost.

In the simulations with the multi-driver and multi-evader model, the main idea of optimal controls are similar with the one-driver and one-evader case. The differences mainly comes from the gathering of evaders, while choosing the escaping direction and driving the pursuit dynamics are similar as we may expect from Figure \ref{fig:multi}. 

\subsection{The optimal control problem for the stabilization of the herd}\label{sec:sim_herding}

Until now, we have considered the problem as a guidance, in the sense that we focus on how to drive the evaders into the target position while ignoring the final velocities of the evaders.

On the other hand, it is also natural to design control strategies to capture the evaders in a small region. If we have one or two drivers, then the final position of the evader cannot be stabilized with the pursuit motion of the drivers. Instead, the drivers need to rotate around the evader, so that it can present rodeo behavior to trap the evaders. When there are more than two drivers, they can make a surrounding formation to keep the evaders between them. 

When the number of drivers is large, we also need to restrict the drivers' positions. If we can steer the evaders only with few drivers, then the optimization algorithm makes redundant drivers move out of the evaders' sight.
In order to enhance the cooperation of drivers, we use the following cost function:

\begin{equation*}
\begin{aligned}
J =& \frac{1}{N} \sum_{i=1}^N \int_0^{t_f} |\bu_{ei}(t) - \bu_f|^2 dt + \delta_3 \int_0^{t_f} \frac{1}{N} \sum_{i=1}^N |\bv_{ei}(t)|^2 dt \\
&+ \delta_3 \int_0^{t_f} \frac{1}{M} \sum_{j=1}^M ( |\bu_{dj}(t) - \bu_f|^2+ |\bv_{dj}(t)|^2)dt+ \frac{\delta_1}{M} \sum_{j=1}^M \int_0^{t_f} |\kappa_j^c(t)|^2 dt,
\end{aligned}
\end{equation*}
According to the cost function, both the drivers and evaders need to be stabilized near the point $\bu_f$ with small velocities. In order to exclude redundant oscillations, we set the friction coefficients to be $10$.

Here we may use a fixed and large final time $t_f>0$ since we trap the evaders in an area. This optimization algorithm takes more time compared to the previous simulations, from tens of minutes to the order of hours. 

\subsubsection{Controls with one or two drivers}

Figure \ref{fig:simus8-0} and \ref{fig:simus8-1} shows two simulations with one driver and two drivers. The regularization coefficients are $\delta_1 = \delta_3 = 0.01$. To simulate the optimal controls, we set initial guess as constant controls which steer the evader in the direction toward $(1,1)$. Until the evader is driven near the desired point, the control functions show a similar pattern as in the previous simulations, Figure \ref{fig:simus6} and \ref{fig:simus7}.

\begin{figure}[]
  \centering
  {
    \includegraphics[width=0.7\textwidth]{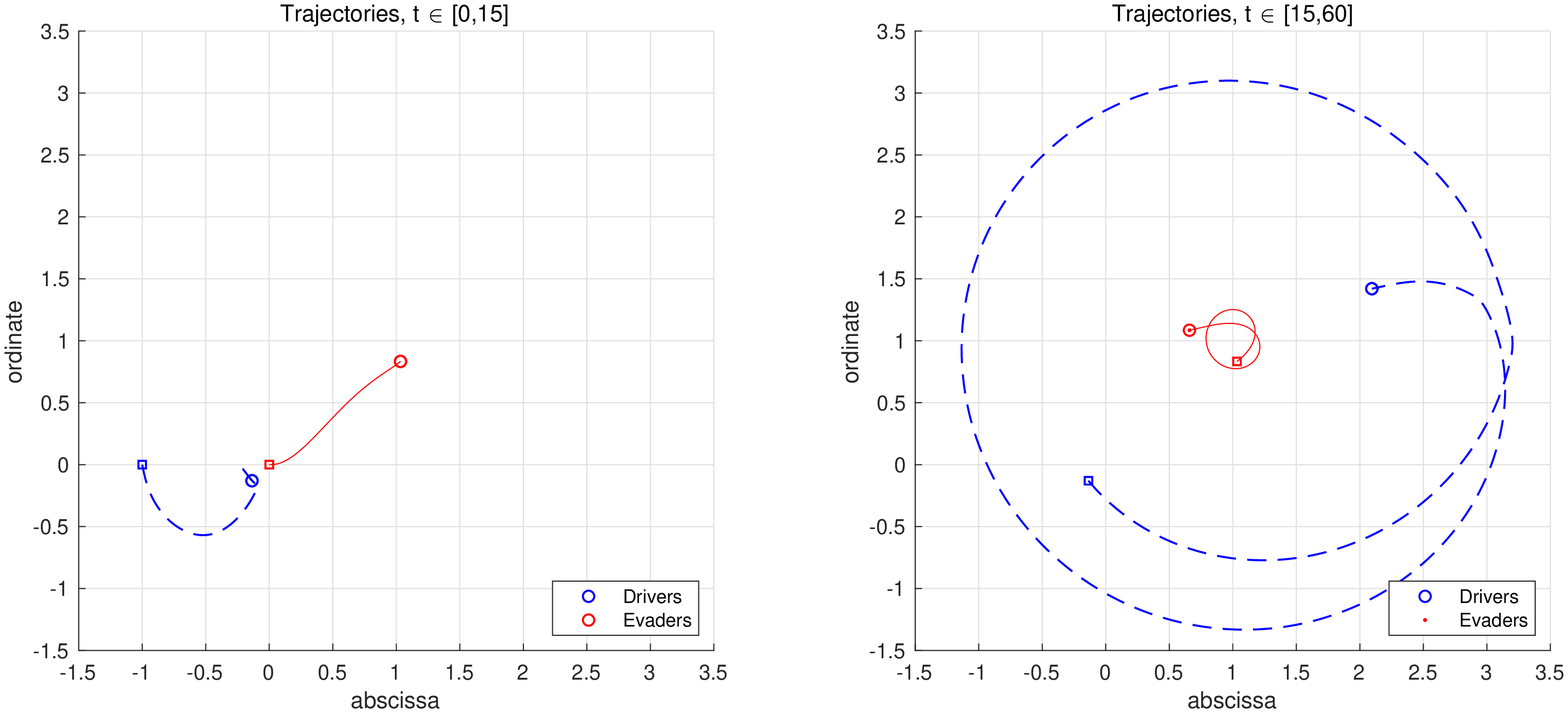}
  }
  \centering
  {
    \includegraphics[width=0.7\textwidth]{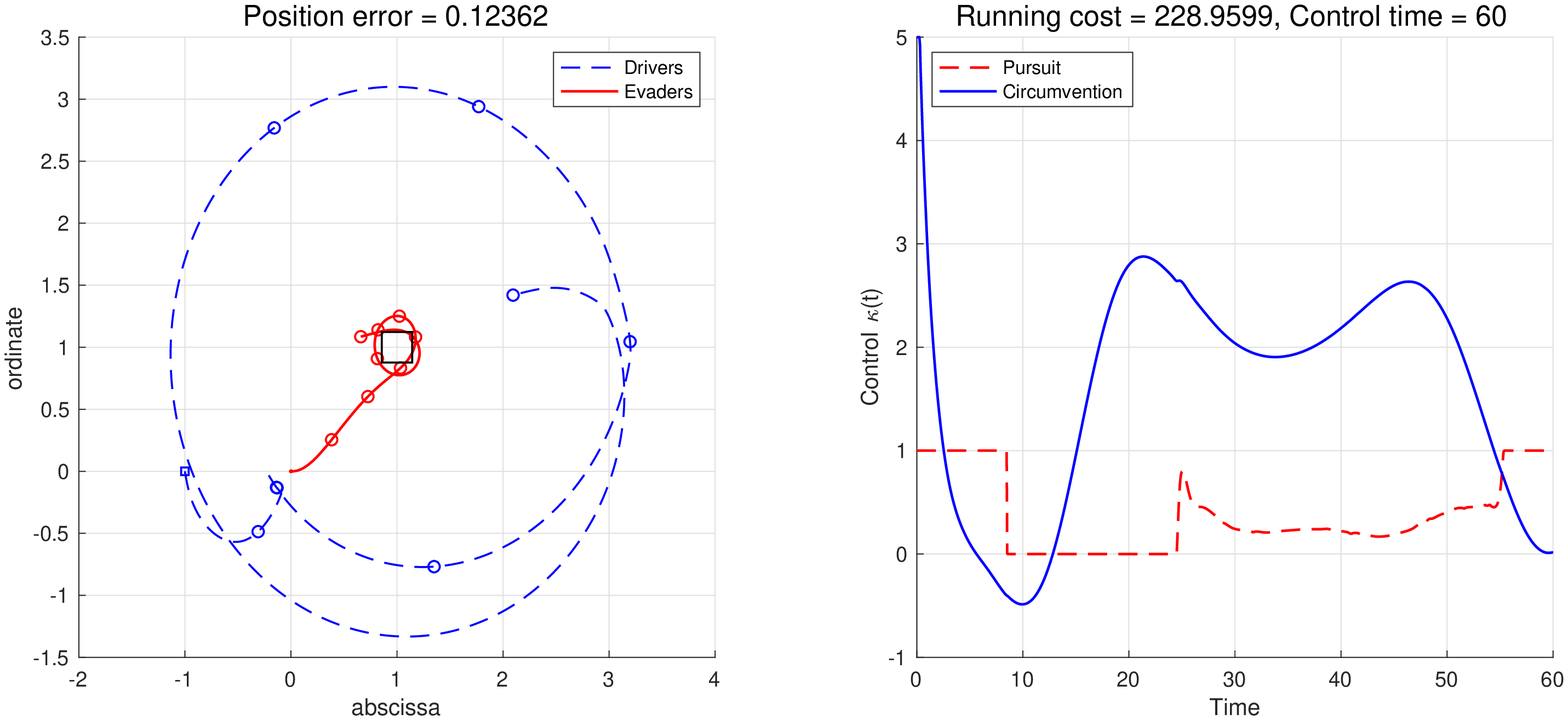}
  }
  \caption{Diagrams for the control leading to $(1,1)$ with one driver and one evader. The drivers' trajectories are presented in blue dashed lines where the square and circle marks show the previous and current positions, respectively. The two plots above present the trajectories of the drivers and evader over the subintervals of time $t \in [0,10]$ and $[10,60]$.}
  \label{fig:simus8-0}
\end{figure}

\begin{figure}[]
  \centering
  {
    \includegraphics[width=0.7\textwidth]{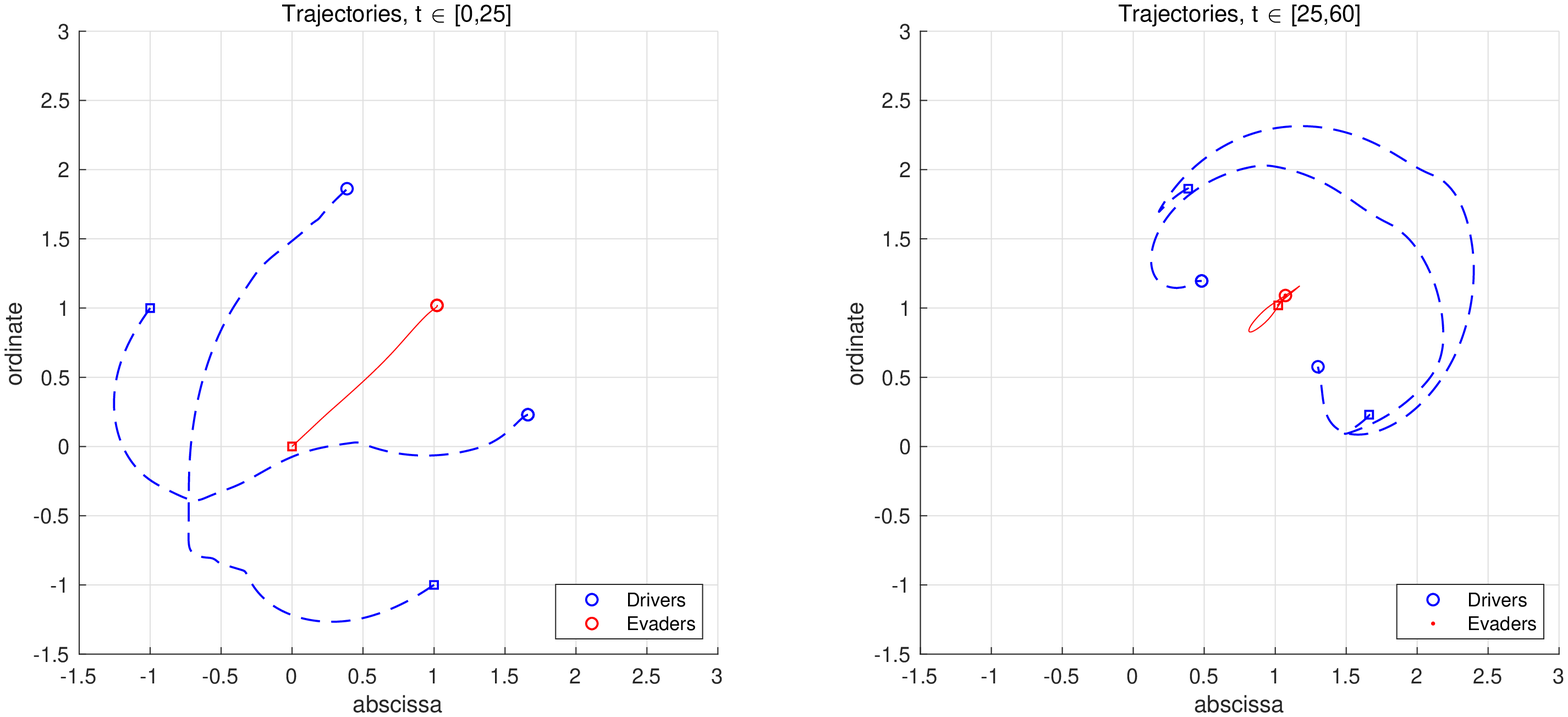}
  }
  \centering
  {
    \includegraphics[width=0.7\textwidth]{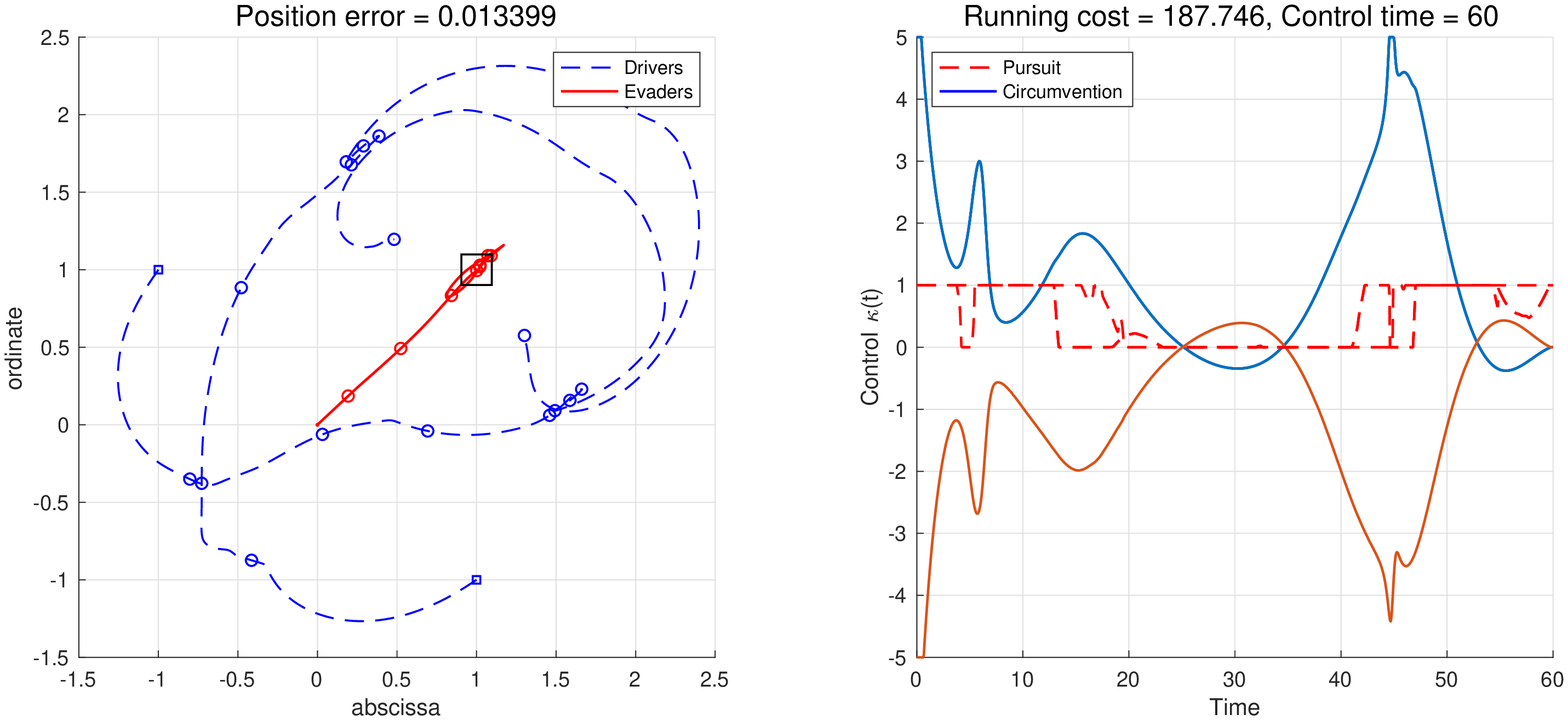}
  }
  \caption{Diagrams for the control leading to $(1,1)$ with 2 drivers and one evader. The two plots above present the trajectories of the drivers and evader over the subintervals of time $t \in [0,25]$ and $[25,60]$.}
  \label{fig:simus8-1}
\end{figure} 

 In those two simulations, Figure \ref{fig:simus8-0} and \ref{fig:simus8-1}, we can observe the rodeo behavior of the drivers after driving the evader near the target point. 
 
 In Figure \ref{fig:simus8-0}, the initial positions are $(-1,0)$ and $(0,0)$ for the driver and evader, respectively, and the target point is at $(1,1)$. In the beginning, the driver repels the evader until $t = 9$, and then it stops tracking until $t = 15$ while the evader moves to $(1,1)$.
 After that, the driver rotates around the evader, so that it can not escape from the target point. This simulation is done until $t = 60$, where the driver can approach the evader near the final time in terms of the cost function. The final position error is $0.12362$, which is in a reasonable distance from $(1.1)$. 

The same phenomenon happens in Figure \ref{fig:simus8-1}. The drivers starts from the positions $(-1,1)$ and $(1,-1)$ while the other initial data are the same as in Figure \ref{fig:simus8-0}. After the time $t = 25$, the evader is nearly stabilized until $t = 35$, but it slightly moves toward the northeast direction.
 When the evader tries to escape the desired point, the drivers rotate with strong circumvention controls to block the evader. The final position error is $0.01340$, which is small enough with respect to the cost function. Note that the running cost ($187.7$) dramatically decreased compared to that ($229.0$) of Figure \ref{fig:simus8-0} since two drivers can cooperate to trap the evader between them.

\subsubsection{Controls with 4 driver and one evader}

Figure \ref{fig:simus8} shows a simulation of 4 drivers with one evader. The initial positions of the drivers are given around the evader at $(0,0)$, namely, $(1,0)$, $(0,1)$, $(-1,0)$ and $(0,-1)$. As before, the drivers want to move the evader into the target point $(1,1)$. The regularization coefficients are $\delta_1 = \delta_3 = 0.1$ from now on, since it is easy to stabilize the evader with 4 drivers.

\begin{figure}[]
  \centering
  {
    \includegraphics[width=0.7\textwidth]{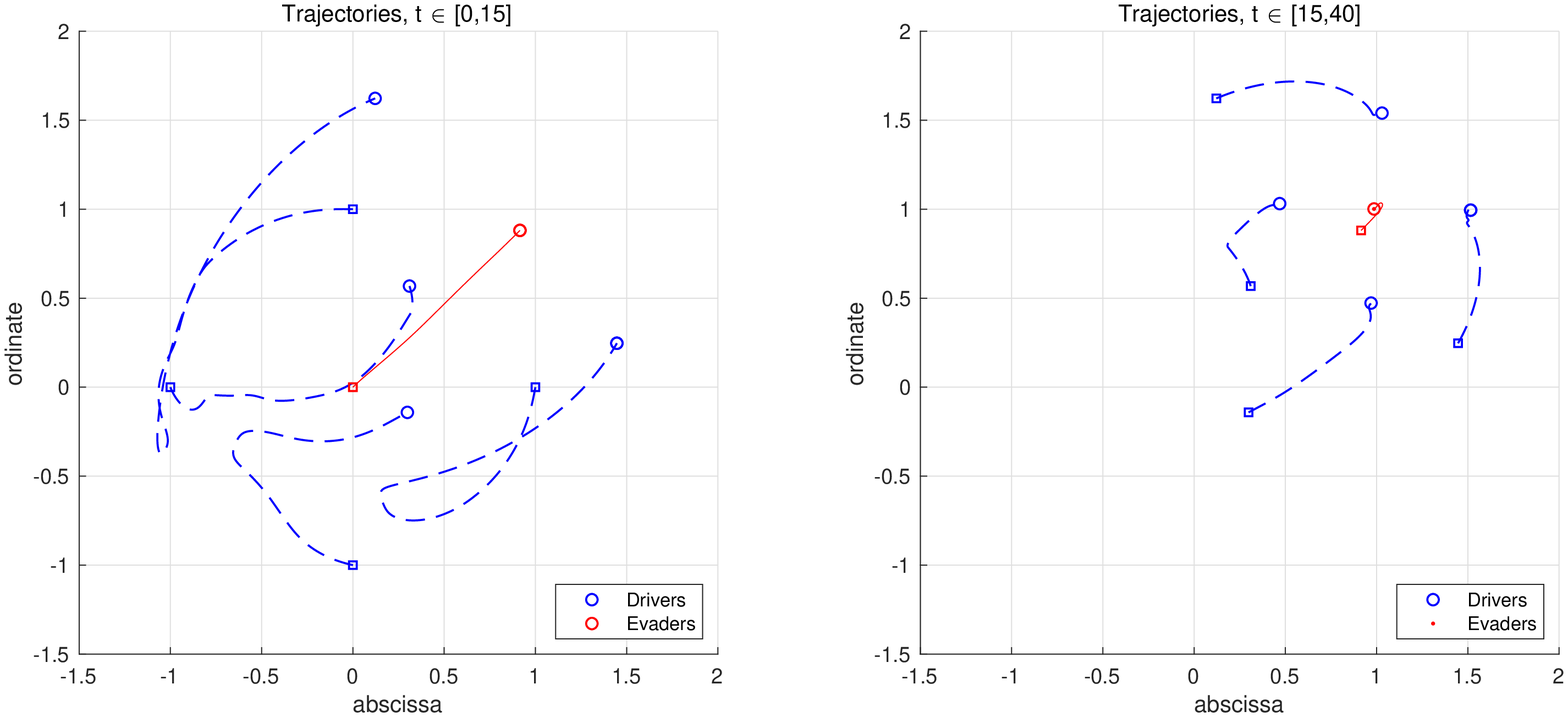}
  }
  \centering
  {
    \includegraphics[width=0.7\textwidth]{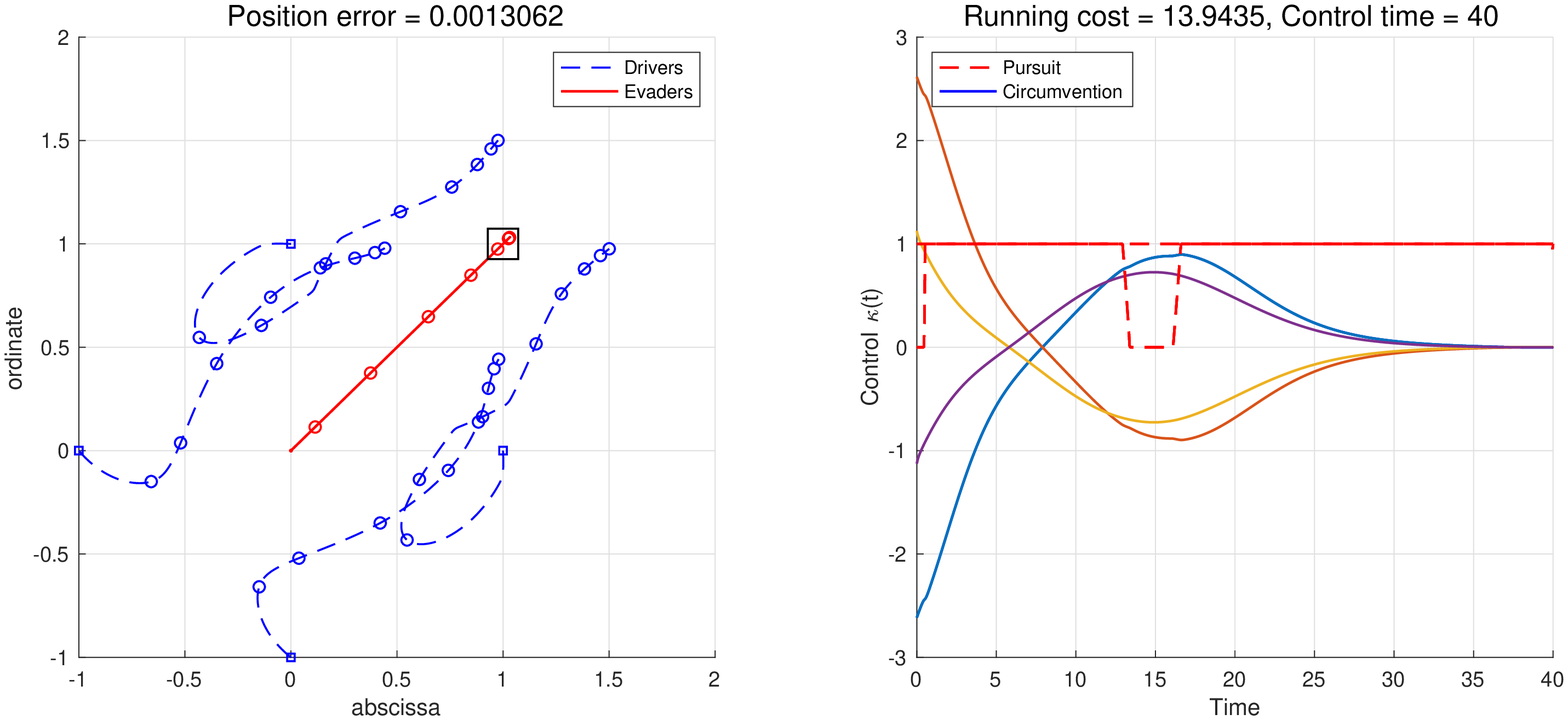}
  }
  \caption{Diagrams for the control leading to $(1,1)$ with 4 drivers and one evader. The two plots below present the trajectories of the drivers and evader over the subintervals of time $t \in [0,15]$ and $[15,40]$.}
  \label{fig:simus8}
\end{figure}

As we already saw with one or two drivers, the control trajectories are similar to Figure \ref{fig:simus6} and \ref{fig:simus7} until the evader is driven. Then, two forward drivers move faster than the evader after that time, and stabilize the evader's position with small circumvention controls. 
The final position error is small enough, $0.001306$, compared to the running cost, $13.94$. Compared to the rodeo behavior in Figure \ref{fig:simus8-0} and \ref{fig:simus8-1}, the running cost is significantly reduced since they don't need to rotate much to trap the evader.

\subsubsection{Controls with the flocking of evaders}

This control strategy works similar even in the case of many evaders. Figure \ref{fig:simus9} presents a simulation with 4 drivers and 16 evaders. The initial positions of evaders are set to be in $[-0.2,0.2]^2$ with a uniform distribution. 
The control functions are much complicated in this simulation, but from the trajectories, we can observe the same pattern as in Figure \ref{fig:simus8}. The circumvention controls are exaggerated in the driving process to rotate with a safe distance since the evaders may separate if they are too close. The same phenomena happen when there are more than $16$ evaders, but the rotational distance should be modified according to the interaction $\psi_d$.

\begin{figure}[]
  \centering
  {
    \includegraphics[width=0.7\textwidth]{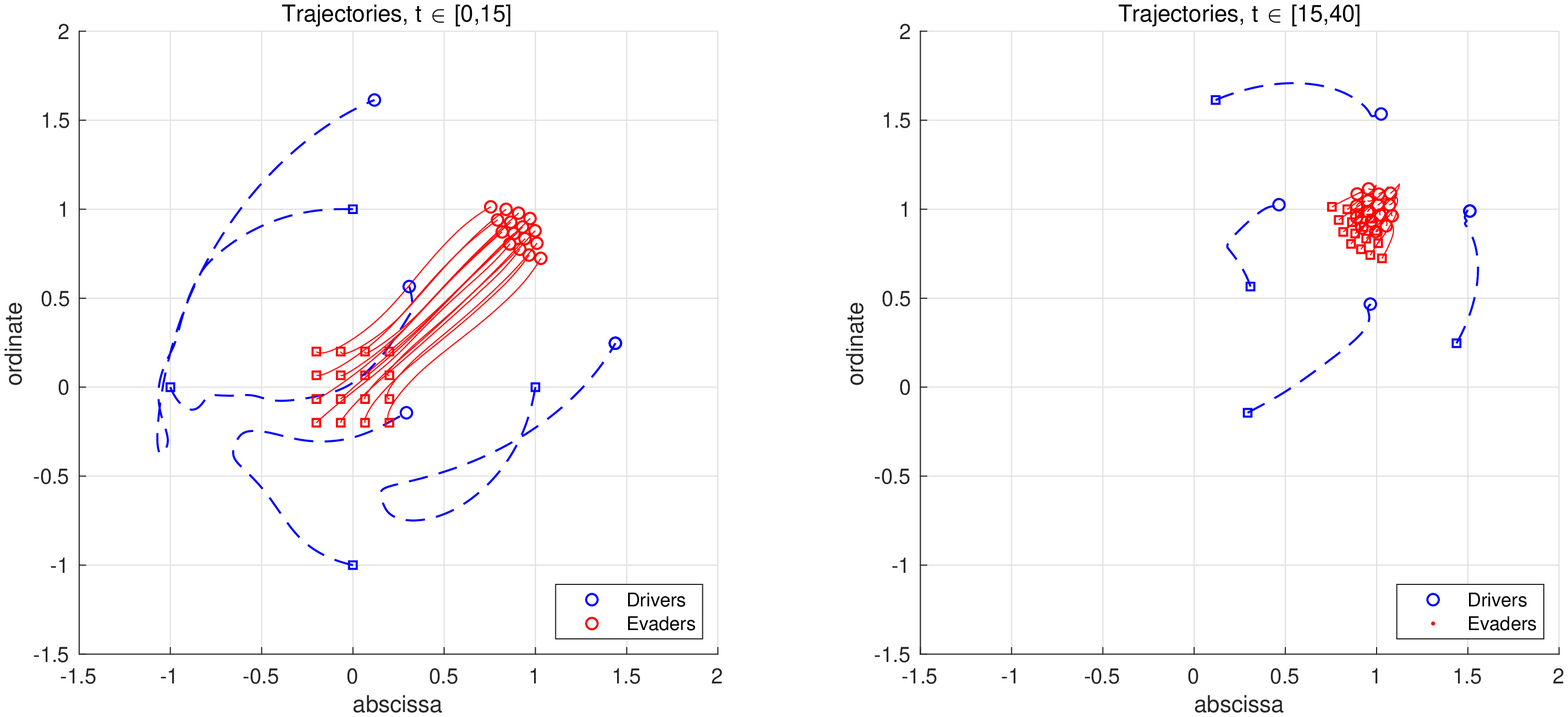}
  }
  \centering
  {
    \includegraphics[width=0.7\textwidth]{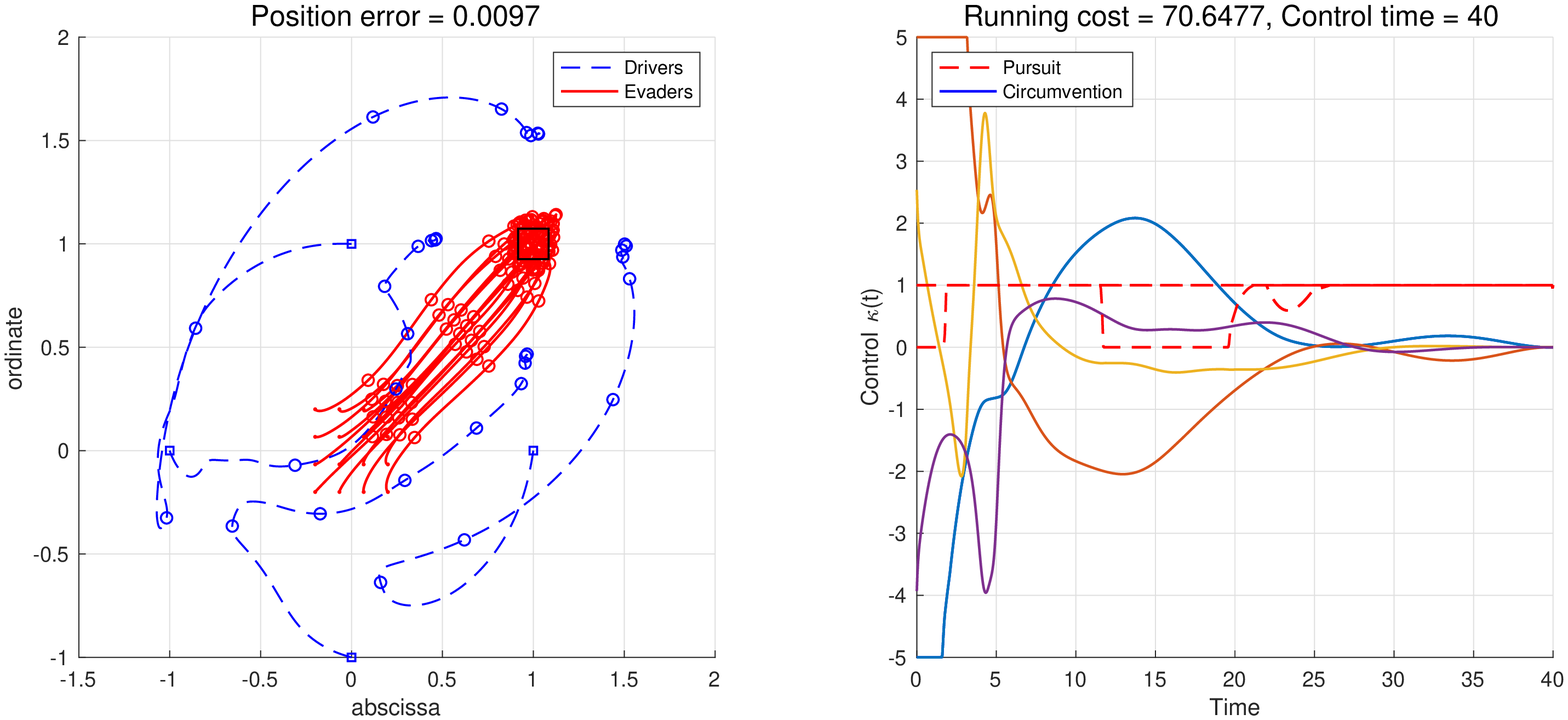}
  }
  \caption{Diagrams for the control leading to $(1,1)$ with 4 drivers and 16 evaders. The two plots below present the trajectories of the drivers and evader over the subintervals of time $t \in [0,15]$ and $[15,40]$.}
  \label{fig:simus9}
\end{figure}

\subsubsection{Controls without interactions among evaders}

The optimal control solution becomes much complicated when there is no flocking behavior of the evaders. Figure \ref{fig:simus10} present a simulation with the same problem as in Figure \ref{fig:simus9} without the interactions among evaders, 
\[ \psi_e(r) \equiv 0. \]
The circumvention controls are bigger than in Figure \ref{fig:simus9} to properly gather the evaders. Note that the trajectories are similar to the previous simulation, Figure \ref{fig:simus9}, where the difference comes from the diameter of evaders. We can observe that the pursuit controls are changing rapidly to make a distance between the drivers and evaders, which reduces the separation of evaders. 

\begin{figure}[]
  \centering
  {
    \includegraphics[width=0.7\textwidth]{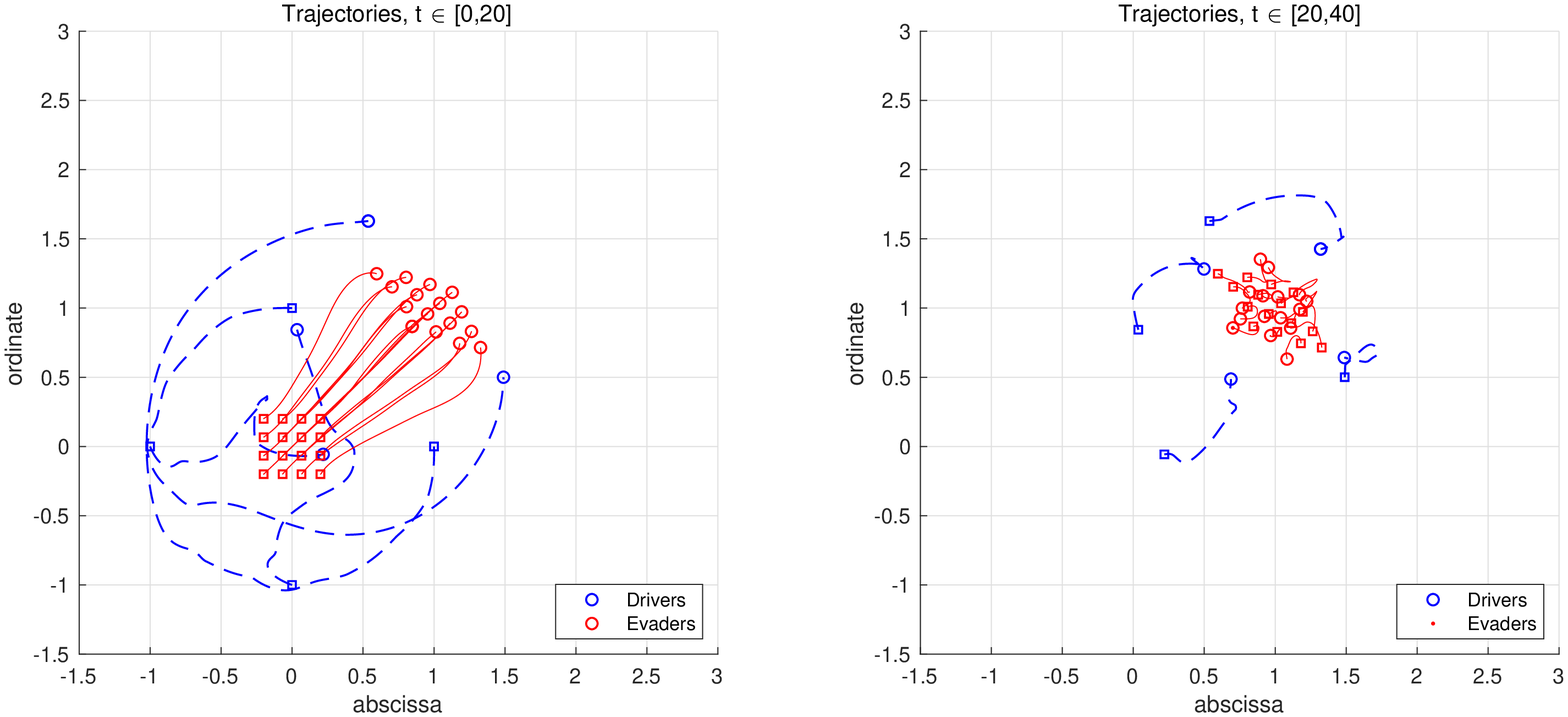}
  }
  \centering
  {
    \includegraphics[width=0.7\textwidth]{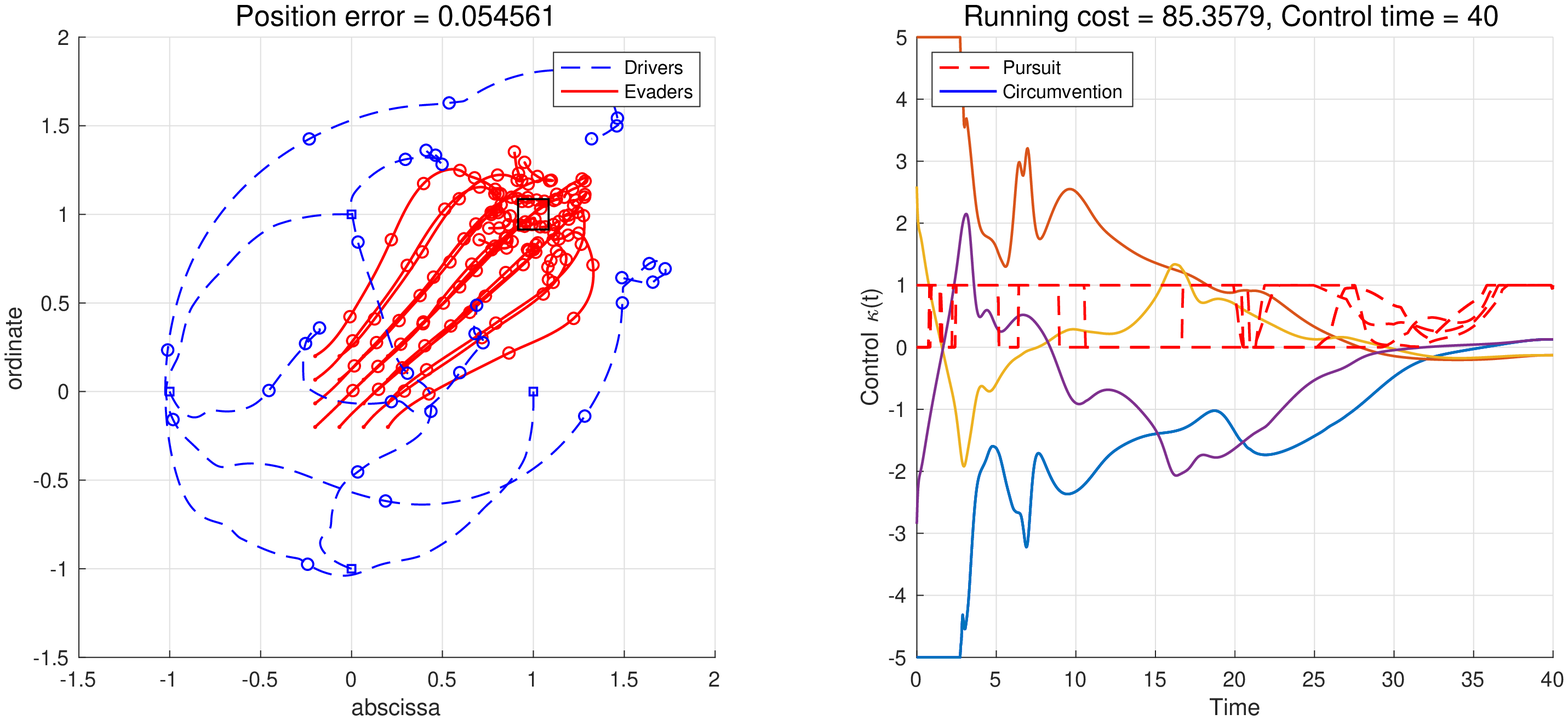}
  }
  \caption{Diagrams for the control leading to $(1,1)$ with 4 drivers and 16 evader without any interaction among evaders. The two plots below present the trajectories of the drivers and evader over the subintervals of time $t \in [0,20]$ and $[20,40]$.}
  \label{fig:simus10}
\end{figure}

\vspace{1em}
\section{Feedback control to herd the evaders}
\label{sec:feedback}

We observed that the optimal strategies in Section \ref{sec:multi} follow a common strategy, first rotate (to fix the escaping direction) and then drive (with small circumvention controls). This also coincides with the idea of off-bang-off control, shown in Theorem \ref{T2.2}. 
In this section, we will construct feedback control functions based on the idea of Section \ref{sec:multi}. we first consider how to steer the escaping direction of one evader using the circumvention control $\kappa^c_j$, and then discuss the multi-evader case with the pursuit control $\kappa^p_j$.

\subsection{Feedback control to steer the direction}

First, we assume the pursuit control $\kappa^p_j(t)$ is constantly $1$ and consider the one-evader case. In order to drive the evader in a proper direction, we need to construct a circumvention control from the current positions of the driver and evader with respect to the target point $\bu_f$.

Figure \ref{fig:feedback0} shows the simulations of \eqref{GBR_simple} on one evader with the feedback control:
\begin{equation}\label{feedback1}
\kappa^c_j(t) = -\bar\kappa^c\frac{(\bu_f - \bu_{ec})\cdot(\bu_{dj}-\bu_{ec})^\perp}{|\bu_f - \bu_{ec}|\cdot|\bu_{dj}-\bu_{ec}|},\quad \bar\kappa^c = 3,\quad j=1,2,\cdots.
\end{equation}
where $|\kappa^c_j(t)| \leq \bar\kappa^c$ and the control strategy is independent of $j$. The parameter $\bar\kappa^c = 3$ is chosen from the maximal value of the control function in Figure \ref{fig:simus4}. Since the circumvention control is bounded, the problem is well-posed from Theorem \ref{T2.1}. 

\begin{figure}[ht]
  \centering
  {
    \includegraphics[width=0.7\textwidth]{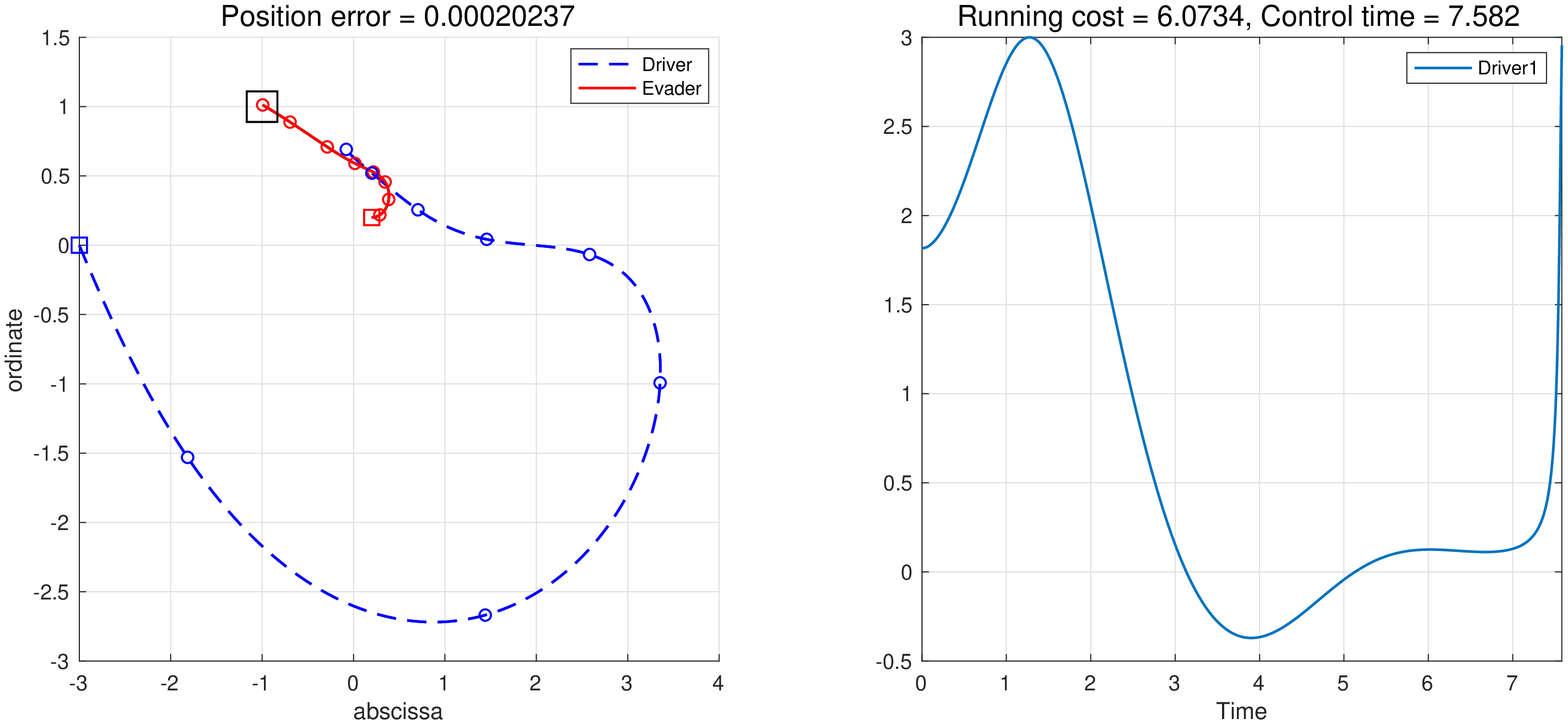}
  }
  \centering
  {
    \includegraphics[width=0.7\textwidth]{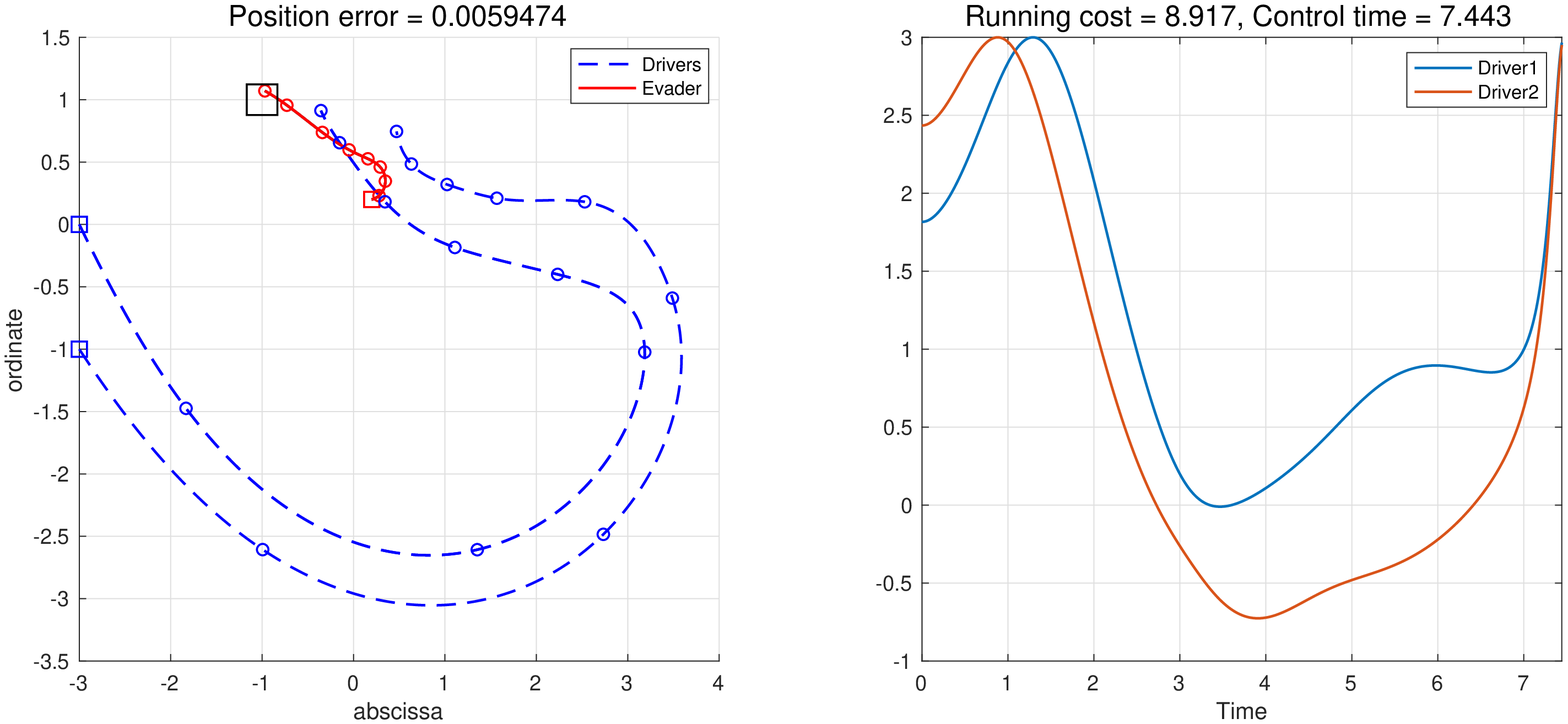}
  }
  \caption{Trajectories (left) and control functions (right) for the feedback control \eqref{feedback1} on the system \eqref{GBR_general} under the same conditions with Figure \ref{fig:simus3} and Figure \ref{fig:simus5}.}
  \label{fig:feedback0}
\end{figure}

The box marks represents the initial positions, where we have the same initial conditions and control costs as in Figure \ref{fig:simus3} and \ref{fig:simus5}. Note that the feedback control \eqref{feedback1} makes similar costs ($6.0734$) with the optimal solutions($4.8546$ and $6.9923$), where the control time ($7.582$) also lies between the control times ($8.447$ and $6.9884$) of Figure \ref{fig:simus3} and \ref{fig:simus5}. The same happens in the model with two drivers and one evader.

The feedback control $\kappa^c_j(t)$ is designed to increase the angle between $(\bu_f - \bu_{ec})$ and $(\bu_{dj} - \bu_{ec})$ so that the driver gets the position behind the target. For example, if the $j$-th driver $\bu_{dj}$ is on the left side of the target $\bu_f$ (from $\bu_{ec}$), then the control $\kappa^c_j(t)$ is positive. Then, the driver rotates to the counterclockwise direction and then the driver can push the evaders to the target $\bu_f$. 
The strength $\bar \kappa^c$ determines how far the driver should rotate away from the center $\bu_{ec}$. After the driving time $t = 7.582$, the vector $(\bu_f - \bu_{ec})$ will direct nearly the opposite direction as $(\bu_{dj} - \bu_{ec})$, so that the driver will make a big rotation (according to $\bar\kappa^c$) to steer the evader again.

\subsection{Feedback control to gather evaders}

Even though the evaders are gathered initially, the pressure from the drivers may violate the flocking of the evaders. In Figure \ref{fig:feedback1}, the trajectories of the system and other indicators are shown in the simulation with $16$ evaders, leading to the target $(4,4)$. We can observe that the diameter of evaders' position increases in the pursuing dynamics. This is from the central interaction between the driver and the evader, where the evaders get closer in the direction of $(\bu_{dj} - \bu_{ec})$ but separates with respect to $(\bu_{dj} - \bu_{ec})^\perp$.

\begin{figure}[ht]
  \centering
  {
    \includegraphics[width=0.7\textwidth]{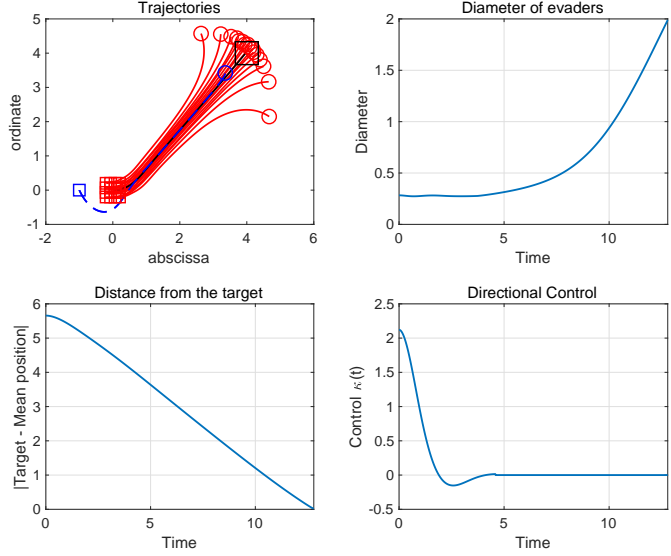}
  }
  \caption{Trajectories, diameter, distance and control of the system \eqref{GBR_general} with $16$ evaders and the feedback control \eqref{feedback1}, leading to $\bu_{ec}(t_f) \simeq (4,4)$ with 16 evaders. Blue, red and black lines indicate the driver, the evaders and the barycenter of evaders, respectively.}
  \label{fig:feedback1}
\end{figure}

For the process of gathering, we adopt a stopping strategy for the drivers as in the release mode. If we set $\kappa_j^p(t)$ to be zero, then the drivers stop the accelerations and the evaders escape away to get enough distance from the drivers. After that, they get close each other and flock again.

Figure \ref{fig:feedback3} shows the trajectory with the feedback law
\begin{equation}\label{feedback2}
\kappa^p_j(t) = \begin{cases} 0 \quad\text{if }~ \max_i|\bu_{ei} - \bu_{ec}| > 0.3 \text{ and until it goes below } 0.27, \\ 1 \quad \text{otherwise}, \end{cases}
\end{equation}
where the values $0.27$ and $0.3$ are determined by hand according to the number of evaders, the interaction $\psi_e$, and the friction $\nu$.

\begin{figure}[ht]
  \centering
  {
    \includegraphics[width=0.7\textwidth]{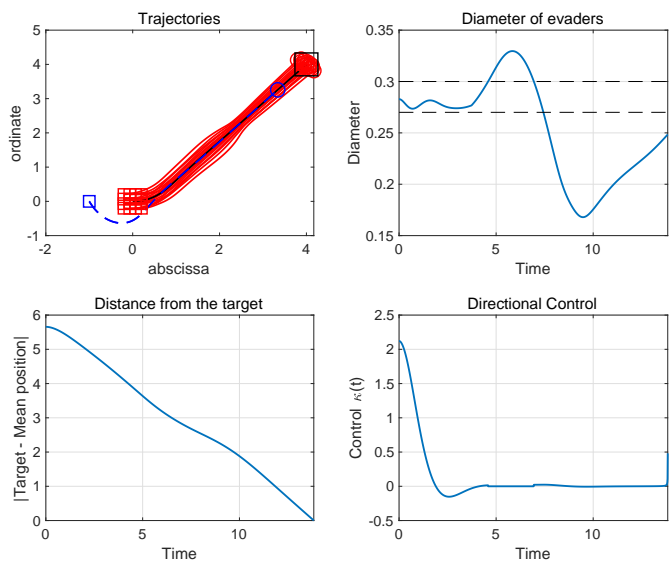}
  }
  \caption{Trajectories, diameter, distance and control of the system \eqref{GBR_general} with the feedback \eqref{feedback1} and \eqref{feedback2} and the same target as in Figure \ref{fig:feedback1}.}
  \label{fig:feedback3}
\end{figure}

Note that the diameter of evaders are not increasing much when the driver stops in Figure \ref{fig:feedback3}. The dynamics of diameter depends on the inertial and friction values, where $\nu=2$ is enough to see the diameter stabilized. The simulation works robust when the position of evaders are initially gathered and the friction coefficients are large.

\begin{figure}[ht]
  \centering
  {
    \includegraphics[width=0.7\textwidth]{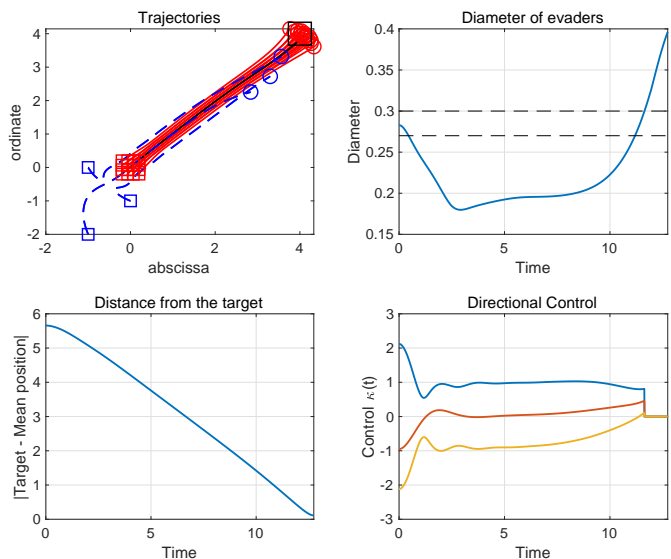}
  }
  \caption{Trajectories, diameter, distance and control of the system \eqref{GBR_general} with the feedback \eqref{feedback1} and \eqref{feedback2}, where there are $3$ drivers from the initial positions $(-1,0)$, $(-1,-2)$ and $(0,-1)$.}
  \label{fig:feedback4}
\end{figure}

The feedback controls \eqref{feedback1} and \eqref{feedback2} also work properly with many drivers, as we can see in Figure \ref{fig:feedback4}. Though we applied the same control function \eqref{feedback1} for each driver, they have different trajectories and values of controls from their different positions and interactions between them. Since the drivers are separated in the perpendicular direction toward the target, the diameter of evaders is more stable compared to Figure \ref{fig:feedback3}. Note that the diameter exceeds the reference value $0.3$ near the time $t=12$, while it was near $t=5$ in Figure \ref{fig:feedback3}.

\vspace{1em}
\section{Conclusions and open problems}\label{sec:con}

The guidance-by-repulsion model consists of repulsive interactions between two types of agents, evaders and drivers, where we want to herd the evaders by manipulating the trajectories of the drivers. This control problem is not classical since the goal is to see whether the evaders' trajectories pass near a given point. Moreover, since the control focuses on the perpendicular direction, it is a kind of bilinear control problem. 

In a sequence of simulations, we showed that the multi-driver and multi-evader system is close to the problem with one driver and one evader, especially for the driving process to the target. The drivers take places behind the target, and push the evaders with nearly linear trajectories.

To analyze it in a closer look, we simplified the model to the case of one driver and one evader, with the same friction coefficients. We present globally stable steady states for constant control functions, and the off-bang-off control turns out to be able to steer the evader's position.
However, the arguments we used relies on the Lyapunov functions from linear theory, which are not easy to extend to general cases.
It makes several interesting questions on the herding problem.

\subsection{General friction coefficients}
We assumed the same dissipation assumption that the frictions in the model \eqref{GBR_general} take the same value $\nu$. In this way, we can simplify the asymptotic motion in terms of the relative position. 

Of course, different coefficients are desirable in real applications. With small perturbations of dissipations, simulations show nearly the same dynamical properties. However, even in the linear model, finding the necessary and sufficient condition for the the stability is a difficult problem.

Moreover, when the sheep or dogs expect the next movement of others, the effect of anticipation may generate nonlinear damping effect on the equations \cite{shu2019anticipation}. When we consider large number of sheep, it is natural to consider this since it also enhance the gathering of sheep.

\subsection{Global boundedness of the relative distance}
Unfortunately, only finite-time boundedness is shown in Theorem \eqref{T2.1}, when the control functions are not constant. This is from the technical difficulties of the second order potential dynamics with a source term, which can not be observed in the gradient flow or the first order model.

In the linearized model, we may build a Lyapunov function which grows quadratically and monotonically nonincreasing near infinity. Also for the opposite case, we can built a Lyapunov function which is infinity at zero and monotonically nonincreasing near zero. However, this can not guarantee the uniform boundedness since the relative distance may oscillate from small values to the large values with increasing amplitude.

This is related to the collision-avoidance problem in a collective behavior model, whether there exists a lower bound of the relative distances between particles for a given unbounded interaction kernel.

\subsection{The reference position of evaders}
For the multi-evader system, we assumed that the drivers interact with the barycenter of the evaders. This makes all the drivers follow the same point in spite of their current positions. However, it also causes strange behavior, for example, we need a fast rotational motion for drivers to escape from the ensemble of the evaders.
 For a practical model, we need to determine the sight of each driver and a reasonable tracking point, which is a nontrivial problem.

\subsection{Feedback control which actively gathers the evaders}
In order to gather the evaders, in the feedback control \eqref{feedback2}, we used $\kappa^p_j(t)$ to stop the pursuit of the drivers and wait the natural flocking behavior of the evaders. This works well when the evaders are relatively close to each other since the interaction of evaders rapidly decreases along the distances.

 However, if the ensemble of evaders is separated and hard to flock together initially, then this strategy does not work. In this situation, we need a rotational motion around the evaders to gather the evaders in an active way. This strategy can be observed in the optimal control as in Figure \ref{fig:simus9}-\ref{fig:simus10}, but hard to construct feedback control in the model \eqref{GBR_general}.

\bibliographystyle{acm}
\bibliography{biblio.bib}

\end{document}